\documentclass{birkjour}
 \newtheorem{thm}{Theorem}[section]
 
 \newtheorem{lem}[thm]{Lemma}
 \newtheorem{st}{Statement}
 
 \theoremstyle{definition}
 
 \theoremstyle{remark}

 \numberwithin{equation}{section}

\begin{document}

%
%
%
%
%
%
%
%
%

\title[The compactness of Moser-Trudinger type inequalities in the unit ball]
 {The compactness of Moser-Trudinger\\type inequalities in the unit ball}

\author[Qi Xia]{Qi Xia\textsuperscript{1,2,*}}

\address{%
\textsuperscript{1} School of Mathematical Sciences,
Dalian University of Technology, Dalian, 116024, China\\
\textsuperscript{2} DUT-BSU Joint Institute,
Dalian University of Technology, Dalian, 116024, China}

\email{xq19991231@mail.dlut.edu.cn}

\noindent \textbf{Corresponding author:} Qi Xia

\thanks{ This work is partially supported by the NSF of China under grant Nos. 12001466 and  U19A2079.}
\author[Yufeng Lu]{Yufeng Lu\textsuperscript{1}}

\address{\textsuperscript{1} School of Mathematical Sciences, Dalian University of Technology, Dalian, 116024, China}

\email{luf@dlut.edu.cn}


\date{May 181, 2026}

\begin{abstract}
In this paper, we employ the concentration-compactness principle to show that if the Dirichlet norm is replaced by the standard Sobolev norm, then the supremum of 
\begin{flalign*}
	\int_{\mathbb{B}} |x|^{N \epsilon} \Phi \left( \alpha_N \left( 1+ \epsilon \right) |u|^{ \frac{N}{N-1} } \right) dx 
\end{flalign*}
over all such functions is uniformly bounded. Furthermore, we also prove the existence of extremals. Finally, we consider the compactness of the sequence of extremals of the inequalities and the limit of this sequence is the extremal of
\begin{flalign*}
	\int_{\mathbb{B}} \Phi \left( \alpha_N |u|^{ \frac{N}{N-1} } \right) dx
\end{flalign*}
in $C^1 \left( \mathbb{B} \right)$, where $\alpha_N = N \omega_{N-1}^{ \frac{1}{N-1} } $, $\Phi \left( t \right) := e^t - \sum_{j=0}^{N-2} \frac{t^j}{j!} $ and $\omega_{N-1}$ is the surface of the unit ball in $\mathbb{R}^N$.
\end{abstract}

\maketitle

\section{Introduction} 
Throughout the paper $\mathbb{B}$ is the unit sphere in $\mathbb{R}^N$, $N \geq 2$, and $\omega_{N-1}$ denotes the surface of $\mathbb{B}$. We also have that $W_0^{1,N} \left( \mathbb{B} \right)$  denotes the usual completion of $C_0^{\infty} \left( \mathbb{B} \right)$ in $W^{1,N} \left( \mathbb{B} \right)$ with norm $|\left| \nabla \cdot |\right|_{L^N} := \left( \int_{ \mathbb{B} } | \nabla \cdot |^N dx \right)^{ \frac{1}{N} }$; it is well known that
\begin{flalign*}
	 & W_0^{1,N} \left( \mathbb{B} \right) \not\hookrightarrow L^{\infty} \left( \mathbb{B} \right),\\
	 & W_0^{1,N} \left( \mathbb{B} \right) \hookrightarrow L^p \left( \mathbb{B} \right),~ \forall p \in [1, + \infty).
\end{flalign*}
In \cite{17} and \cite{18}, N. S. Trudinger and J. Moser proved that there exists $\alpha > 0$ such that $W_0^{1,N} \left( \mathbb{B} \right)$ is embedded
in the Orlicz space $L_{\phi_{\alpha}} \left( \mathbb{B} \right)$ determined by the Young function $\phi_{\alpha} \left( t \right) = e^{ \alpha |t|^{ \frac{N}{N-1} } } - 1$, and there exist a sharp constant $\alpha_N = N \omega_{N-1}^{ \frac{1}{N-1} }$ and positive constants $C$ such that 
\begin{flalign*}
	\underset{u \in W_0^{1,N} \left( \mathbb{B} \right) \setminus \{0\}, |\left| \nabla u |\right|_{ L^{N} } = 1 }{\sup} \int_{\mathbb{B}} e^{ \alpha |u|^{ \frac{N}{N-1} } } dx \leq C |\mathbb{B} |, ~ \forall \alpha \leq \alpha_N.
\end{flalign*}
Also, the supremum above is $+\infty$ if $\alpha > \alpha_N$. Moreover, for any bounded domain
$\Omega$, the constant $C$ can be attained, and we refer to \cite{22, 23, 24, 25}. Also, more extensive results about the Moser-Trudinger inequality on compact Riemannian manifolds can be found in several papers, such as \cite{29, 30, 31}.

On the crucial question of compactness for the embedding $W_0^{1,N} \left( \Omega \right) \hookrightarrow L_{ \phi_{\alpha} } \left( \Omega \right)$, P. L. Lions proved that the embedding $W_0^{1,N} \left( \mathbb{B} \right) \hookrightarrow L_{\phi_{\alpha}} \left( \mathbb{B} \right)$ is compact when positive constants are in small weak neighborhoods of 0 with bounded domain $\Omega$ in \cite{19}.
\begin{st}
	Let $\{ u_k \}_k$ be a sequence of functions in $W_0^{1,N} \left( \mathbb{B} \right)$ with $ |\left| \nabla u_k |\right|_{L^N} =1$ such that $u_k \rightharpoonup u_0 \neq 0$ in $W_0^{1,N} \left( \mathbb{B} \right)$. Then we have 
	\begin{flalign*}
		\underset{k}{\sup} \int_{\mathbb{B}} e^{ \alpha_N p |u_k|^{ \frac{N}{N-1} } } dx < \infty,
	\end{flalign*}
	for any $0 < p < \frac{1}{ \left( 1- |\left| \nabla u_0 |\right|_{L^N}^N \right)^{ \frac{1}{N-1} } }$.
\end{st}
It is clear that this result gives more precise information than the classical type Moser-Trudinger inequality when $u_k \rightharpoonup u_0$ in
$W_0^{1,N} \left( \Omega \right)$ with $u_0 \neq 0$. In \cite{9}, an interesting extension of Moser-Trudinger inequality is to construct Trudinger-Moser type
inequalities on bounded domains with Dirichlet norm or unbounded domains $\Omega$ including $\mathbb{R}^N$ with norm as follows,
\begin{flalign*}
	|\left| \cdot |\right| : = \left( \int_{\Omega} |\nabla \cdot|^N + |\cdot|^N dx \right)^{ \frac{1}{N} }.
\end{flalign*}
Let
\begin{flalign*}
	\Phi \left( t \right) = e^t - \sum_{j=0}^{N-2} \frac{t^j}{j !},
\end{flalign*}
that was ﬁrst considered by
D. M. Cao in $\mathbb{R}^2$ in \cite{10} and for any dimensions by J. M. B. do \'{O} in \cite{11}. Moreover, Y. Li and B. Ruf proved the following results in \cite{12},
\begin{flalign*}
	\underset{u \in W_0^{1,N} \left( \Omega \right) \setminus \{0\}, |\left| u|\right| \leq 1}{\sup} \int_{\Omega} \Phi \left( \alpha |u|^{ \frac{N}{N-1} } \right) dx < \infty,~ \alpha\leq \alpha_N.
\end{flalign*}
Also, $\alpha_N$ is a sharp constant, that is, the supremum is $+\infty$ for any $\alpha > \alpha_N$. On the other hand, in \cite{20, 21, 28}, M. de Souza and others considered the following type inequality,
\begin{flalign*}
	\underset{ u \in W^{1,N} \left(\mathbb{R}^N \right) \setminus \{0\}, |\left| u |\right| \leq 1 }{\sup} \int_{\mathbb{R}^N} \Psi \left( \alpha_N \left( 1 + \beta |\left| u|\right|_{L^N}^N \right)^{ \frac{1}{N-1} } |u|^{ \frac{N}{N-1} } \right) dx < \infty,
\end{flalign*}
for any $0 \leq \beta <1$, where
\begin{flalign*}
	\Psi \left( t \right) = e^t - \sum_{j=0}^{N-1} \frac{t^j}{j!} = \Phi \left( t \right) - \frac{t^{N-1} }{\left( N-1 \right)!}.
\end{flalign*}

Inspired by the principle due to P. L. Lions and others, M. de Souza and others also establish the following
improvement of the Trudinger–Moser inequality as follows in \cite{9},
\begin{st}
	Let $\{u_k\}_k$ be a sequence in $W^{1,N} \left( \mathbb{R}^N\right)$ such that $|\left| u_k |\right|=1$ and $u_k \rightharpoonup u_0 \neq 0$ in $W^{1,N} \left( \mathbb{R}^N\right)$. If
	\begin{flalign*}
		0 < p < p_N \left( u_0 \right) : = \frac{1}{\left( 1 - |\left| u_0 |\right|^N \right)^{\frac{1}{N-1}  }  },
	\end{flalign*}
    then
    \begin{flalign*}
   	    \underset{k}{\sup} \int_{ \mathbb{R}^N } \Phi \left( \alpha_N p u_k \right) dx < \infty.
    \end{flalign*}
    Moreover, $ p_N \left( u_0 \right)$ is sharp.
\end{st} 
From the rearrangement argument and Young inequality, Y. Yang extended this type of Moser-Trudinger inequality in the presence of singular potentials in \cite{13}, 
\begin{flalign*}
	\underset{u \in W^{1,N} \left( \mathbb{R}^N \right) \setminus \{0\}, \int_{\mathbb{R}^N} |\nabla u|^N + \tau |u|^N dx \leq1 }{\sup} \int_{ \mathbb{R}^N } \frac{ \Phi \left( \alpha_N \left( 1- \beta \right) |u|^{ \frac{N}{N-1} }  \right) }{ |x|^{ N \beta } }  dx < \infty,
\end{flalign*}
for any $0 \leq \beta <1$ and $\tau >0$. Yang also established the existence of extremals by using blow-up analysis in \cite{14}. 

Recently, Y. Wang and Y. Yang proved the compactness of extremals for the critical singular
Trudinger–Moser inequality in $\mathbb{R}^2$ in \cite{16}. It is shown that up to a subsequence, $\{ u_{\beta} \}_{\beta}$ converges to some
functions $u_0$ in $C^1 \left( \Omega \right)$ as $\beta \rightarrow 0^+$. Also, X. Li and others considered the compactness of extremals for it in \cite{15} and proved that the sequence of functions converges to extremal of
\begin{flalign*}
	\underset{u \in W^{1,N} \left( \mathbb{R}^N \right) \setminus \{0\}, \int_{\mathbb{R}^N} |\nabla u|^N + \tau |u|^N dx \leq1 }{\sup} \int_{ \mathbb{R}^N } \Phi \left( \alpha_N |u|^{ \frac{N}{N-1} }  \right)  dx
\end{flalign*}
in $C^1 \left( \mathbb{R}^N \right)$.  From a geometric point of view, we let $\omega_0 : = \sum_{j=1}^N d x_j^2$ be the standard Euclidean metric, and $\omega_{\epsilon} = |x|^{N \epsilon} \omega_0$ be the conical metric. Then we have $d \nu_{\omega_{\epsilon}} = |x|^{N \epsilon} dx$ and $\omega_{\epsilon} \rightarrow \omega_0$ as $\epsilon \rightarrow 0^+$ in $C^2_{loc} \left( \mathbb{B} \setminus \{0\} \right)$. 

In this paper, we will extend these results to more common Moser-Trudinger type inequalities on the unit ball. Then we consider the Moser-Trudinger inequality with conical metric as follows,
\begin{flalign*}
	\underset{u \in W_0^{1,N} \left( \mathbb{B} \right) \setminus \{0\}, |\left| u|\right| \leq 1}{\sup} \int_{\mathbb{B} } |x|^{N \epsilon} \Phi \left( \alpha \left( 1+ \epsilon \right)  |u|^{ \frac{N}{N-1} } \right) dx,
\end{flalign*} 
where $\epsilon > 0$. Firstly, we consider the existence of extremal functions and the sharp constant; while we consider the compactness of Moser-Trudinger functionals and the limit of it. Then, we use this result and the method mentioned in \cite{9} to show the concentration-compactness principle in $\mathbb{B}$, which can be used to prove the existence of extremal functions of the Moser-Trudinger inequality in the first main result as follows,
\begin{thm}[Theorem 1.1]
	There exist constants C such that, for the unit ball $\mathbb{B} \subset \mathbb{R}^N$,  
	\begin{flalign*}
		\underset{u \in W_0^{1,N} \left( \mathbb{B} \right) \cap \mathcal{S} \setminus \{0\}, ||\nabla u||_{L^N} \leq 1 }{\sup} \int_{\mathbb{B}} |x|^{N \epsilon} \Phi \left( \alpha_N \left(1 + \epsilon  \right) |u|^{ \frac{N}{N-1} } \right)dx \leq C.
	\end{flalign*} 
	For any $\alpha $ and any $u \in W_0^{1,N} \left( \mathbb{B} \right) \cap \mathcal{S} \setminus \{0\}$, where $\mathcal{S}$ is the set of all radially symmetric functions, there holds that
	\begin{flalign*}
		\int_{\mathbb{B}} |x|^{N \epsilon} \Phi \left( \alpha \left(1 +\epsilon \right) |u|^{ \frac{N}{N-1} } \right) dx < \infty. 
	\end{flalign*}
	Moreover, the inequality is sharp, where $\alpha_N = N \omega_{N-1}^{\frac{1}{N-1}}$, and $\Phi(t) = e^t - \sum_{n=0}^{N-2} \frac{t^n}{n !}$. Also, the supremum of this inequality can be attained by some functions. 
\end{thm}

Moreover, we consider whether a maximizer sequence $\{ u_{\epsilon} \}_{\epsilon}$ converges or not when the conical metric converges to the Euclidean metric. Then we state the second main result as follows,

\begin{thm}[Theorem 1.2]
	Assume that $\{u_{\epsilon} \}_{\epsilon}$ is a sequence of maximizers for the supremum in Theorem 1.1. Then up to a subsequence, there exists some function $u_0$ satisfying $u_{\epsilon} \rightarrow u_0$ in $C^1 \left( \overline{\mathbb{B}} \right)$ and $u_0$ is an extremal function of the following supremum   
	\begin{flalign*}
		\underset{u \in W_0^{1,N} \left( \mathbb{B} \right) \cap \mathcal{S} \setminus \{0\}, ||\nabla u||_{L^N} \leq 1 }{\sup} \int_{\mathbb{B}} \Phi \left( \alpha_N |u|^{ \frac{N}{N-1} } \right)dx.
	\end{flalign*}
\end{thm}

The remaining part of this paper is organized as follows. In Section 2, we display several
key estimates to prove the concentration-compactness principle in Lemma 2.3. While Theorem 1.1 is proved in Section 3, and Theorem 1.2 is proved by blow up analysis from \cite{26, 27} in Section 4.

In order to prove the second part of Theorem 1.1., we prove that, for any fixed constant $\alpha > \alpha_N (1 + N \epsilon)$, there exist sequences $\{ u_{n} \}_{n} \subset W_0^{1,N} \left( \mathbb{B} \right) \cap \mathcal{S} \setminus \{0\}$ with $|| \nabla u_{n}||_{L^N} \leq 1$ such that 
\begin{flalign*}
	\underset{n \rightarrow \infty}{\lim} \int_{ \mathbb{B} } |x|^{N \epsilon} \Phi \left( \alpha |u_n|^{ \frac{N}{N-1} } \right) dx = + \infty.
\end{flalign*}
From Lions' Lemma, we get $u_n \rightharpoonup 0$. Then by the compact embedding theorem, we may assume that $||u_n||_{L^p} \rightarrow 0$ for any $p>1$. Consequently, $|| \nabla u_n||_{L^N} \rightarrow 1$, and 
\begin{flalign*}
	\alpha \left( \frac{u_n}{||\nabla u_n||_{L^N} } \right)^{ \frac{N}{N-1} } > \beta u_n^{ \frac{N}{N-1} },
\end{flalign*} 
where $\beta \in (\alpha_N \left( 1 + N \epsilon \right), \alpha)$. Then we have that
\begin{flalign*}
	\underset{n \rightarrow \infty}{\lim} \int_{\mathbb{B}} |x|^{N \epsilon} \Phi \left( \alpha \left( \frac{ |u_n| }{ ||\nabla u_n ||_{L^N} } \right)^{ \frac{N}{N-1} }  \right) dx & \geq \underset{n \rightarrow \infty}{\lim} \int_{\mathbb{B}} |x|^{N \epsilon} \Phi \left( \beta |u_n|^{ \frac{N}{N-1} }  \right) dx\\
	& \geq 	\underset{n \rightarrow \infty}{\lim} \int_{\mathbb{B}} |x|^{N \epsilon} \left( e^{ \beta |u_n|^{ \frac{N}{N-1} } } - 1 \right) dx\\
	& = + \infty. 
\end{flalign*}
The first part of the Theorem 1.1 and the Theorem 1.2  will be proved by blow up analysis. 

From defination of $\Phi$, we have that
\begin{flalign*}
	- \operatorname{div} \left( |\nabla u|^{N-2} \nabla u \right) = \frac{1}{\lambda^{\phi}} |x|^{N \epsilon} |u|^{ \frac{1}{N-1} } \Phi^{\prime} \left( \beta |u|^{ \frac{N}{N-1} }\right),
\end{flalign*}
where 
\begin{flalign*}
	\lambda^{\phi}_{\epsilon} = \int_{\mathbb{B}} |x|^{N \epsilon} |u|^{ \frac{N}{N-1} } \Phi^{\prime} \left( \beta |u|^{ \frac{N}{N-1} } \right) dx.
\end{flalign*}
In addition, from the maximum principle and elliptic theory, we have that the functions $u$ mentioned above satisfy $u>0$ in $\mathbb{B}$ and $u=0$ on $\partial \mathbb{B}$.

\section{Some Lemmas}

\begin{lem} \textbf{\cite{3}} \label{lem: 2.1}
	Let $u \in W^{1,N} \left( \mathbb{R}^N \right)$, and let $u^*$ be the Schwarz symmetrization of $u$. If $||\nabla u^*||_{L^N} \leq 1$, then for
	each $R > 0$ and $q > 0$, there exists a positive constant $C = C(q,R, N)$ such that
	\begin{flalign*}
		\int_{|x| \geq R} \Phi \left( q |u^*|^{ \frac{N}{N-1} } \right) dx \leq C(q,R,N).
	\end{flalign*}
\end{lem}

\begin{lem} \textbf{\cite{2}} \label{lem: 2.2}
	Let $u \in W_0^{1,N} \left( \mathbb{B}_R \right)$ be such that $ ||\nabla u||_{L^N} \leq 1$. If $u^*$ is the decreasing rearrangement
	of $u$, then
	\begin{flalign*}
		u^* (t) \leq \left\lbrace \frac{1}{N \omega_{N-1}^{ \frac{1}{N-1} } }\ln \left( \frac{| \mathbb{B}_R| }{t} \right) \right\rbrace^{ \frac{N-1}{N} },
	\end{flalign*}
    for all $ t \in \left(0, |\mathbb{B}_R| \right)$.
\end{lem}

\begin{lem} \textbf{\cite{4}} \label{lem: 2.3}
	We consider the function $\zeta$: $\mathbb{N} \times \mathbb{R} \rightarrow \mathbb{R}$,
	\begin{flalign*}
		\zeta \left(N,t \right) : = e^t - \sum_{n=0}^{N-2} \frac{t^n}{n!}.
	\end{flalign*}
	Let $t \geq 0$, $p \geq 1$ be real numbers, and let $N \geq 2$ be an integer. Then there holds
	\begin{flalign*}
		\left( \zeta \left( N, t \right) \right)^p \leq \zeta \left( N, p t \right).
	\end{flalign*}
\end{lem}

\begin{lem} \label{lem: 2.4}
	Let $\{ u_k \}_k \subset W_0^{1,N} \left( \mathbb{B} \right)$ be a sequence such that $||u_k||_{W_0^{1,N}} =1$ and $u_k \rightharpoonup u_0\neq 0$ in $W_0^{1,N} \left( \mathbb{B} \right)$. If
	\begin{flalign*}
		0 < p < p(u_0) := \frac{1}{\left(1 - ||\nabla u_0||_{L^N}^N \right)^{ \frac{1}{N-1} } },
	\end{flalign*}
    then
    \begin{flalign*}
    	\underset{k}{\sup} \int_{ \mathbb{B} } \Phi(\alpha_N p |u_k|^{ \frac{N}{N-1} }) dx < \infty.
    \end{flalign*}
\end{lem}

\begin{proof}
	\textbf{Case 1.} In this case, from $u_k \rightharpoonup u_0$ in $W_0^{1,N} \left( \mathbb{B} \right)$ and the uniformly convex
	Banach space, we have that $u_k \rightarrow u_0$ in $W_0^{1,N} \left( \mathbb{B} \right)$. From Theorem 49 in \cite{1}, up to a subsequence, we have $|u_k (x)| \leq v(x)$ for almost every $x \in \mathbb{B}$ and some $v(x) \in W_0^{1,N} \left( \mathbb{B} \right)$. Hence, the
	proof follows from the Trudinger–Moser inequality and the Lebesgue dominated convergence
	theorem.
	
	\textbf{Case 2.} If $ 0 < ||\nabla u_0||_{L^N} <1$. Suppose, for contradiction, that for some $0 < p < \frac{1}{ \left( 1 - ||\nabla u_0||_{L^N}^N\right)^{ \frac{1}{N-1} } }  $ we have
	\begin{flalign*}
		\underset{k}{\sup} \int_{ \mathbb{B} }  \Phi \left(\alpha_N p |u_k|^{ \frac{N}{N-1} }  \right) dx = + \infty,
	\end{flalign*}
    which implies that  
    \begin{flalign*}
    	\underset{k}{\sup} \int_{ \mathbb{B} }  \Phi \left(\alpha_N p |u_k^*|^{ \frac{N}{N-1} }  \right) dx = + \infty,
    \end{flalign*}
    where $u_k^*$ is the Schwarz symmetrization of $u_k$. Fixed $0 < R < 1$, we let that
    \begin{flalign*}
    	\int_{\mathbb{B}} \Phi \left(\alpha_N p |u_k^*|^{ \frac{N}{N-1} }  \right) dx = \int_{\mathbb{B}_R} \Phi \left(\alpha_N p |u_k^*|^{ \frac{N}{N-1} }  \right) dx + \int_{\mathbb{B} \setminus \mathbb{B}_R } \Phi \left(\alpha_N p |u_k^*|^{ \frac{N}{N-1} }  \right) dx.
    \end{flalign*}
    From $W_0^{1,N} \left( \mathbb{B} \right) \hookrightarrow L^s \left( \mathbb{B} \right)$ and $s \geq 1$, we have that
    \begin{flalign*}
    	\int_{ \mathbb{B} } \sum_{n=0}^{N-2} |u_k^*|^{ \frac{n N}{N-1} } dx < \infty.
    \end{flalign*}
    From this estimate and Lemma \ref{lem: 2.1}, we get that
    \begin{flalign*}
    	\underset{k}{\sup} \int_{\mathbb{B}_R} \Phi \left(\alpha_N p |u_k^*|^{ \frac{N}{N-1} }  \right) dx = + \infty.
    \end{flalign*}
    From classical estimates, we have two inequalities as follows,
    \begin{flalign*}
    	\left( s+ t \right)^{ \frac{N}{N-1} } \leq s^{ \frac{N}{N-1} } + A s^{ \frac{1}{N-1} } t + t^{ \frac{N}{N-1} },
    \end{flalign*}
    for some constants $A=A(N) >0$; and all $s$, $t \geq 0$, 
    \begin{flalign*}
    	s^{q} t^{ q^{\prime} } \leq \epsilon s + \epsilon^{ \frac{-q}{ q^{\prime} } } t,
    \end{flalign*} 
    for $\epsilon >0$, $q+ q^{\prime} =1$ and $q$, $q^{\prime} > 0$.
    
    Let $\{v_k \}_k \subset W_0^{1,N} \left( \mathbb{B}_R \right)$ be such that $ v_k (x) = u^*_k(x) - u^*_k (R)$. Then we have 
    \begin{flalign*}
    	|u^*_k (x)|^{ \frac{N}{N-1} } & \leq \left( |v_k (x)| + |u^*_k(R)| \right)^{ \frac{N}{N-1} }\\ 
    	& \leq |v_k (x)|^{ \frac{N}{N-1} } + A |v_k (x)|^{ \frac{1}{N-1} } |u^*_k(R)| + |u^*_k(R) |^{ \frac{N}{N-1} },
    \end{flalign*}
    \begin{flalign*}
     	|v_k (x)|^{ \frac{1}{N-1} } |u^*_k(R)| & = \left( |v_k (x)|^{ \frac{N}{N-1} } \right)^{ \frac{1}{N} } \left( |u^*_k (R) |^{ \frac{N}{N-1} } \right)^{ \frac{N-1}{N} } \\
     	&\leq \frac{\epsilon}{A} | v_k (x)|^{ \frac{N}{N-1} } + \left( \frac{\epsilon}{A} \right)^{ - \frac{1}{N-1} } |u^*_k (R)|^{ \frac{N}{N-1} }.
    \end{flalign*}
    In addition, we have
    \begin{flalign*}
    	|u^*_k(x)|^{ \frac{N}{N-1} } \leq \left( 1+ \epsilon \right) | v_k^* (x)|^{ \frac{N}{N-1} } + C (N, \epsilon) |u^*_k(R)|^{ \frac{N}{N-1} }.
    \end{flalign*} 
    We can find $\epsilon >0$ and $p_1 >0$ such that
    $(1 + \epsilon) p_1 p < p \left(u_0 \right)$, with H\"{o}lder inequality, 
    \begin{flalign*}
    	& \int_{\mathbb{B}_R} e^{ p \alpha_N |u^*_k(x)|^{ \frac{N}{N-1} } } dx \\
    	& \leq \left( \int_{\mathbb{B}_R} e^{ \left( 1+ \epsilon \right) p_1 p \alpha_N |v^*_k(x)|^{ \frac{N}{N-1} } } dx \right)^{ \frac{1}{p_1} } \left( \int_{\mathbb{B}_R} e^{ p_1^{\prime} p \alpha_N C (N, \epsilon ) |u^*_k(R)|^{ \frac{N}{N-1} } } dx \right)^{ \frac{1}{ p_1^{\prime} } },
    \end{flalign*}
    where $\frac{1}{p_1}+ \frac{1}{p_1^{\prime}} =1$.
    From the theory of rearrangement functions, we have that the second integral on the right hand side is ﬁnite; i.e., 
    \begin{flalign*}
    	\int_{ \mathbb{B}_R } e^{p_1^{\prime} p \alpha_N C \left( N, \epsilon \right) |u_k^* \left( R \right)|^{ \frac{N}{N-1} } } dx < \infty.
    \end{flalign*}
    Thus, we may assume that
    \begin{flalign*}
    	\underset{k}{\sup} \int_{\mathbb{B}_R} e^{ \left( 1+ \epsilon \right) p_1 p \alpha_N |v^*_k(x)|^{ \frac{N}{N-1} } } dx = + \infty,
    \end{flalign*}
    where $v^*_k$ is the decreasing rearrangement of $v_k$. Since the defination of $v_k$ and the
    P\'{o}lya–Szeg\"{o} inequality, we have
    \begin{flalign*}
    	|| \nabla v_k^*||_{L^N} \leq ||\nabla v_k||_{L^N} = || \nabla u_k^*||_{L^N} \leq ||\nabla u_k||_{L^N} \leq 1.
    \end{flalign*}
    From Lemma \ref{lem: 2.2}, we have that
    \begin{flalign}
    	v^* (t) \leq \left\lbrace \frac{1}{N \omega_{N-1}^{ \frac{1}{N-1} } }\ln \left( \frac{| \mathbb{B}_R| }{t} \right) \right\rbrace^{ \frac{N-1}{N} }, \label{eq: 3.1}
    \end{flalign}
    for all $ t \in \left(0, |\mathbb{B}_R| \right)$. We now only need to prove that for each $p_2 \in (p_1 p, p \left(u_0 \right))$, $t_0 \in (0, |\mathbb{B}_R|)$ and each $k \in \mathbb{N}$, there exists $k>k_0$ and $t\in(0, t_0)$ such that
    \begin{flalign*}
    	v^* (t) \geq \left( \frac{1}{p_2 \alpha_N}  \right)^{ \frac{N-1}{N} } \ln^{ \frac{N-1}{N} } \left( \frac{|\mathbb{B}_R|}{t} \right).
    \end{flalign*}
    Indeed, suppose by contradiction that there exist $k_0 \in \mathbb{N}$ and $t_0 \in (0, |\mathbb{B}_R|)$ such that, for some $p_2 \in (p_1 p, p \left(u_0 \right) )$,
    \begin{flalign*}
    	v^* (t) < \left( \frac{1}{p_2 \alpha_N}  \right)^{ \frac{N-1}{N} } \ln^{ \frac{N-1}{N} } \left( \frac{|\mathbb{B}_R|}{t} \right)
    \end{flalign*}
    $k>k_0$ and $t \in(0, t_0)$. Then together with inequality \ref{eq: 3.1}, implies
    \begin{flalign*}
    	\int_0^{|\mathbb{B}_R|} e^{ \left( 1+ \epsilon \right) p_1 p \alpha_N |v^*_k(t)|^{ \frac{N}{N-1} } } dt & \leq \int_0^{t_0} \left( \frac{| \mathbb{B}_R| }{t} \right)^{ \frac{ (1 + \epsilon) p_1 p }{p_2} } dt \\
    	& + \int_{t_0}^{|\mathbb{B}_R|} \left( \frac{| \mathbb{B}_R |}{ t_0 } \right)^{ (1 + \epsilon) p_1 p } dt < \infty,
    \end{flalign*}
    which contradicts and our claim is proved. While, we can consider sequence $\{ t_k\}_k \subset (0, t_0)$ and $t_k \leq \frac{1}{k}$ for all $k \in \mathbb{N}$ and 
    \begin{flalign*}
    	v^*_k (t_k) \geq \left( \frac{1}{p_2 \alpha_N}  \right)^{ \frac{N-1}{N} } \ln^{ \frac{N-1}{N} } \left( \frac{|\mathbb{B}_R|}{t_k} \right).
    \end{flalign*}
    For $L >0$ and $ v \geq 0$, 
    \begin{flalign*}
    	& T^L (v) := \min \{ v, ~ L \},\\
    	& T_L (v) := v - T^L (v).
    \end{flalign*}
    Obviously, $|| T^L (v)|| \rightarrow ||v||$ as $L \rightarrow + 
    \infty$. Then we can find constants $p_3 \in (p_2, p \left(u_0 \right))$ and $L$ large enough such that
    \begin{flalign*}
    	\frac{1 - ||u_0||^N}{1-||T^L (u_0)||^N} > \left( \frac{p_3}{p \left(u_0 \right)} \right)^{ N-1 }. 
    \end{flalign*}
    Then we have that $v_k^* (t_k) \rightarrow + \infty$ as $ k \rightarrow \infty$. Also, there exists constants $r_k \in (t_k, |\mathbb{B}_R| )$ such that $v_k^* (r_k) = L$ for each $k \in \mathbb{N}$. Thus, we have 
    \begin{flalign*}
    	\left( \frac{1}{p_2 \alpha_N}  \right)^{ \frac{N-1}{N} } \ln^{ \frac{N-1}{N} } \left( \frac{|\mathbb{B}_R|}{t_k} \right) - L \leq v_k^{*} ( t_k ) - v_k^{*} (r_k) = \int_{t_k}^{r_k} - \frac{d v_k^*}{dt} dt.
    \end{flalign*}
    With the H\"{o}lder inequality,
    \begin{flalign*}
    	\int_{t_k}^{r_k} - \frac{d v_k^*}{dt} dt  \leq & || - \frac{d v_k^*}{dt} \left( N^{ \frac{N-1}{N}  } \omega_{N-1}^{ \frac{1}{N} } \right) t^{ \frac{N-1}{N} } ||_{L^N (t_k, r_k) } \\
    	& \times \frac{ N^{ \frac{1- N}{N} } }{ \omega_{N-1}^{ \frac{1}{N} } } \ln^{ \frac{N-1}{N} } \left( \frac{|\mathbb{B}_R|}{t_k} \right).
    \end{flalign*}
    From $\frac{|\mathbb{B}_R|}{t_k} \rightarrow + \infty$ as $k \rightarrow \infty$. We have that
    \begin{flalign*}
    	\left( \frac{1}{p_3} \right)^{ \frac{N-1}{N} } \leq || - \frac{d v_k^*}{dt} \left( N^{ \frac{N-1}{N}  } \omega_{N-1}^{ \frac{1}{N} } \right) t^{ \frac{N-1}{N} } ||_{L^N (0, r_k) } + o_k(1).
    \end{flalign*}
    By the deﬁnitions of $T^L$ and $T_L$, one can see that
    \begin{flalign*}
    	\int_{\mathbb{B}_R} |\nabla T^L(v_k)|^N dx + \int_{\mathbb{B}_R} |\nabla T_L (v_k)|^N dx & = \int_{\mathbb{B}_R} |\nabla v_k|^N dx = \int_{\mathbb{B}_R} |\nabla u_k^*|^N dx\\
    	& = \int_{\mathbb{B}_R} |\nabla T^L(u^*_k)|^N dx + \int_{\mathbb{B}_R} |\nabla T_L (u^*_k)|^N dx.
    \end{flalign*}
    From 
    \begin{flalign*}
    	\int_{\mathbb{B}_R} |\nabla T^L(v_k)|^N dx \geq \int_{\mathbb{B}_R} |\nabla T^L(u^*_k)|^N dx,
    \end{flalign*}
    we have
    \begin{flalign} \label{eq: 2.2}
    	\int_{\mathbb{B}_R} |\nabla T_L(v_k)|^N dx \leq \int_{\mathbb{B}_R} |\nabla T_L(u^*_k)|^N dx.
    \end{flalign}
    From the inequality \ref{eq: 2.2} and the P\'{o}lya–Szeg\"{o} inequality, we observe that
    \begin{flalign*}
    	\int_{\mathbb{B}} |\nabla T_L (u_k) |^N dx & \geq \int_{\mathbb{B}} |\nabla (T_L (u_k))^* |^N dx\\
    	& = \int_{\mathbb{B}} |\nabla T_L (u_k^*) |^N dx\\
    	& \geq \int_{\mathbb{B}} |\nabla T_L (v_k) |^N dx\\
    	& = || - \frac{d v_k^*}{dt} \left( N^{ \frac{N-1}{N}  } \omega_{N-1}^{ \frac{1}{N} } \right) t^{ \frac{N-1}{N} } ||_{L^N (0, r_k) }^N.
    \end{flalign*}
    Then we have that
    \begin{flalign*}
    	\left( \frac{1}{p_3} \right)^{N-1}\leq \int_{\mathbb{B}} |\nabla T_L( u_k )|^{ N } dx + o_k (1) .
    \end{flalign*}
    With $u_k = T^L (u_k) + T_L (u_k)$ and $T^L (u_k) \leq u_k$, we have that
    \begin{flalign*}
    	1 = ||\nabla u_k||_{L^N}^N = ||\nabla T^L (u_k)||_{L^N}^N + ||\nabla T_L (u_k)||_{L^N}^N.
    \end{flalign*}
    Then
    \begin{flalign*}
    	||\nabla T^L(u_k)||_{L^N}^N + \left( \frac{1}{p_3} \right)^{N-1} + o_k(1) \leq 1.
    \end{flalign*}
    For each fixed $L>0$, the sequence $\{ T^L (u_k) \}_k \subset W_0^{1,N} \left( \mathbb{B} \right)$ is bounded. Then we assume that
    \begin{flalign*}
    	& T^L \left(u_k \right) \rightharpoonup T^L \left(u_0 \right),~ \text{in} ~ W_0^{1,N} \left( \mathbb{B} \right),\\
    	& T^L \left(u_k \right) \rightarrow T^L \left(u_0 \right),~ \text{a.e} ~ x \in \mathbb{B} .
    \end{flalign*}
    By the lower semicontinuity
    of the norm and the previous inequality we obtain
    \begin{flalign*}
    	p_3 & \geq \frac{1}{\left( 1- \underset{k}{\lim \inf} ||\nabla T^L (u_k)||^N_{L^N} \right)^{ \frac{1}{N-1} } }\\
    	& \geq \frac{1}{\left( 1-  ||\nabla T^L (u_k)||^N_{L^N} \right)^{ \frac{1}{N-1} } }.
    \end{flalign*}
    Moreover,
    \begin{flalign*}
    	p_3 & \geq \frac{1}{\left( 1-  ||\nabla T^L (u_k)||^N_{L^N} \right)^{ \frac{1}{N-1} } }\\
    	& > \frac{p_3}{p \left(u_0 \right) } \frac{1}{\left( 1-  ||\nabla u_k||^N_{L^N} \right)^{ \frac{1}{N-1} } } \\
    	& = p_3,
    \end{flalign*}
    which is a contradiction.
    
    Next, we show that $p \left(u_0 \right)$ is sharp, i.e., there exist function sequences $\{ u_k\}_k \subset W_0^{1,N} \left( \mathbb{B} \right)$ and a fucntion $u_0 \in W_0^{1,N} \left( \mathbb{B} \right)$ such that 
    \begin{flalign*}
    	& ||\nabla u_k||_{L^N} =1,\\
    	& u_k \rightharpoonup u_0 \neq 0,\\
    	& ||\nabla u_0||_{L^N}= \delta <1,\\
    	& \underset{k}{\lim \inf} \int_{\mathbb{B}} \Phi \left( \alpha_N p \left(u_0 \right) |u_k|^{ \frac{N}{N-1} } \right) dx = + \infty. 
    \end{flalign*}
    For a fixed $r>0$, we defined a sequence of functions as follow,
    \begin{flalign*}
    	m_k (x) =
    	\begin{cases}
    		N^{ \frac{1-N}{N} } \omega_{N-1}^{ - \frac{1}{N} } k^{ \frac{N-1}{N} }, ~ \text{if} ~ |x| \in [0, r e^{ -\frac{k}{N} } ], \\
    		N^{ \frac{1}{N} } \omega_{N-1}^{ - \frac{1}{N} } \ln \left( \frac{r}{|x|} \right) k^{ - \frac{1}{N} }, ~ \text{if} ~ |x| \in [ r e^{ -\frac{k}{N} }, r ],\\
    		0,  ~ \text{if} ~ |x| \in [ r ,1].
    	\end{cases}
    \end{flalign*}
    It is obvious that $m_k(x) \in W_0^{1,N} \left( \mathbb{B} \right)$, $||\nabla m_k(x)||_{L^N} =1$ and $m_k \rightharpoonup 0$ in $W_0^{1,N} \left( \mathbb{B} \right)$. Also, we have that $||\nabla m_k||_{L^p} \rightarrow 0$ as $k \rightarrow \infty$ for all $p \geq 1$. On the other hand, we define a function $u$ as follow,
    \begin{flalign*}
    	u(x)=
    	\begin{cases}
    		A, ~ \text{if} ~ |x| \in [0, \frac{2R}{3}],\\
    		3A - \frac{3A}{R} |x|, ~ \text{if} ~ |x| \in [\frac{2R}{3}, R],\\
    		0, ~ \text{if} ~ |x| \in [R,1],
    	\end{cases}
    \end{flalign*}
    for $R:=3r<1$ and $A>0$. We see that $u \in W_0^{1,N} \left( \mathbb{B} \right)$ and the constant $A$ are chosen in such a way that $|\left|\nabla u|\right|_{L^N} = \delta <1$. Deﬁning
    \begin{flalign*}
    	v_k = u + \left( 1- \delta^N \right)^{ \frac{1}{N} } m_k,
    \end{flalign*}
    and using that $\nabla u$ and $\nabla m_k$ have disjoint supports, it follows that 
    \begin{flalign*}
    	|\left|\nabla v_k|\right|_{L^N}^N = |\left|\nabla u|\right|_{L^N}^N + \left( 1- \delta^N \right).
    \end{flalign*}
    Setting $u_k =\frac{v_k}{ |\left|\nabla v_k|\right|_{L^N} }$, then
    \begin{flalign*}
    	& |\left| \nabla u_k|\right|_{L^N} = 1,\\
    	& u_k \rightharpoonup u_0 ~\text{in} ~ W_0^{1,N} \left( \mathbb{B} \right),\\
    	& |\left|\nabla u_0|\right|_{L^N} = \delta.
    \end{flalign*}
    Thus 
    \begin{flalign*}
    	\int_{\mathbb{B} } \Phi \left( \alpha_N p \left(u_0 \right) |u_k |^{ \frac{N}{N-1} } \right) dx & \geq \int_{\mathbb{B}_{r e^{ -\frac{k}{N} }}  } \exp \left\lbrace \alpha_N \frac{|u_k|^{ \frac{N}{N-1} } }{ \left( 1- \delta^N \right)^{ \frac{1}{N-1} } } \right\rbrace dx + C\left(u_0 \right)\\
    	& = \int_{\mathbb{B}_{r e^{ -\frac{k}{N} }}  } \exp \left\lbrace \alpha_N \frac{\left( A + \left( 1 - \delta^N\right)^{ \frac{1}{N} } m_k  \right)^{ \frac{N}{N-1} } }{ \left( 1- \delta^N \right)^{ \frac{1}{N-1} } } \right\rbrace dx + C \left(u_0 \right)\\
    	& = \int_{\mathbb{B}_{r e^{ -\frac{k}{N} }}  } \exp \left\lbrace \alpha_N \left( C_1 + m_k \right)^{ \frac{N}{N-1} } \right\rbrace dx + C \left(u_0 \right)\\
    	& \geq C_2 e^{-\frac{k}{N}} \exp \left\lbrace \alpha_N \left( C+ k^{ \frac{N-1}{N} } \right)^{ \frac{N}{N-1} } \right\rbrace + C \left(u_0 \right) \rightarrow + \infty,
    \end{flalign*}
    for some positive constants $C$, $C_1$, $C_2$ and this concludes the proof of this lemma.
\end{proof}

\section{Proof of Theorem 1.1}
We divide the proof into several steps. Firstly, we show that for any $\alpha>0$ and $u \in W_0^{1,N} \left( \mathbb{B} \right) \cap \mathcal{S} \setminus \{0\}$, the integral as follows is finite. 
\begin{flalign*}
	\int_{\mathbb{B}} |x|^{N \epsilon} \Phi \left( \alpha |u|^{ \frac{N}{N-1} } \right) dx.
\end{flalign*}
Then, we show that, for any $\alpha \leq \alpha_N \left( 1 + \epsilon \right)$, the following holds
\begin{flalign*}
	\underset{u \in W_0^{1,N  }\left( \mathbb{B} \right) \cap \mathcal{S} \setminus \{0\}}{\sup} \int_{\mathbb{B}} |x|^{N \epsilon} \Phi \left( \alpha |u|^{ \frac{N}{N-1} } \right)dx < + \infty.
\end{flalign*} 
In addition, we find a function $u_0 \in W_0^{1,N} \left( \mathbb{B} \right) \cap \mathcal{S} \setminus \{0\}$ satisfying $|| \nabla u_0||_{L^N} =1$ that attains the supremum of the inequality to complete the proof.

\textbf{Step 1}. With H\"{o}lder inequality and Lemma \ref{lem: 2.3}, we have
\begin{flalign*}
	\int_{\mathbb{B}} |x|^{N \epsilon} \Phi \left( \alpha |u|^{ \frac{N}{N-1} } \right) dx & \leq \left( \int_{\mathbb{B}} |x|^{p N \epsilon} dx \right)^{ \frac{1}{p} }  \left( \int_{\mathbb{B}} \left( \Phi \left(\alpha |u|^{ \frac{N}{N-1} } \right) \right)^{p^{\prime}} dx \right)^{ \frac{1}{p^{\prime}} } \\
	& \leq \left( \int_{\mathbb{B}} |x|^{p N \epsilon} dx \right)^{ \frac{1}{p} }  \left( \int_{\mathbb{B}}  \Phi \left(\alpha p^{\prime} |u|^{ \frac{N}{N-1} } \right) dx \right)^{ \frac{1}{p^{\prime}} },
\end{flalign*}
where $\frac{1}{p}+ \frac{1}{ p^{\prime} }=1$. For any fixed $\alpha < \alpha_N \left( 1+ \epsilon \right)$, there exists $p>0$ such that $\alpha p^{\prime} \leq \alpha_N \left( 1+ \epsilon \right)$. Then the two parts of the right-hand side are finite. The case of $\alpha = \alpha_N \left( 1+ \epsilon \right)$ is established by the concentration-compactness principle. Thanks to the Sobolev imbedding Theorem, the step is complete.

\textbf{Step 2}. 
We consider the maximizing sequence of functions. Let $\{ R_k\}_k$ be a sequence such that $R_k \nearrow 1$ and $\beta_k \nearrow \alpha_N \left( 1+ \epsilon \right)$, then there exists a function $u_k$ with $||\nabla u_k||_{L^N} =1$,
\begin{flalign*}
	\int_{\mathbb{B}_{R_k} } |x|^{N \epsilon} \Phi \left( \beta_k |u_k|^{ \frac{N}{N-1} } \right) dx = \underset{u \in W_0^{1,N} \left( \mathbb{B}_{R_k}\right) \cap \mathcal{S} \setminus \{0\}}{\sup} |x|^{N \epsilon} \Phi \left( \beta_k |u|^{ \frac{N}{N-1} } \right) dx.
\end{flalign*}
From
\begin{flalign*}
	v(r) = \left( 1 + \epsilon \right)^{ 1- \frac{1}{N}} u\left( r^{ \frac{1}{1 + \epsilon} } \right),
\end{flalign*}
then we have
\begin{flalign*}
	\int_{\mathbb{B}_{R_k}} |\nabla u|^N dx = \int_{\mathbb{B}_{R_k}} |\nabla v|^Ndx.
\end{flalign*}
By direct calculation, we obtain
\begin{flalign*}
    \int_{\mathbb{B}_{R_k}} |x|^{N \epsilon} \Phi \left( \beta_k |u_k|^{ \frac{N}{N-1} } \right) dx & = \int_{\mathbb{B}_{R_k}} |x|^{N \epsilon} \left( e^{ \beta_k |u_k|^{ \frac{N}{n-1} } } - \sum_{n=0}^{N-2} \frac{ \beta_k^n |u_k|^{n \frac{N}{N-1} }}{n !} \right) dx\\
	& = \frac{1}{1+ \epsilon} \int_{\mathbb{B}_{R_k} } \Phi \left( \frac{\beta_k}{1+ \epsilon} |v_k|^{ \frac{N}{N-1} }\right) dx.
\end{flalign*}
Then we have that
\begin{flalign*}
	& \underset{u \in W_0^{1,N} \left( \mathbb{B}_{R_k} \right) \cap \mathcal{S} \setminus \{0\}, || \nabla u||_{L^N} =1}{\sup} \int_{\mathbb{B}_{R_k}} |x|^{N \epsilon} \Phi \left( \beta_k |u|^{ \frac{N}{N-1} } \right) dx\\
	 & = \underset{u \in W_0^{1,N} \left( \mathbb{B}_{R_k} \right)  \setminus \{0\}, || \nabla u||_{L^N} =1}{\sup} \frac{1}{1+\epsilon}\int_{\mathbb{B}_{R_k}} \Phi \left( \frac{\beta_k}{1 + \epsilon} |u|^{ \frac{N}{N-1} } \right) dx.
\end{flalign*}
Then the Theorem 1.1 is complete.

\section{Proof of Theorem 1.2}
Denoting $c_{\epsilon}:= \underset{ \mathbb{B}}{\max} u_{\epsilon} = u_{\epsilon} (0)$, we assume that $c_{\epsilon} \rightarrow + \infty$ as $\epsilon \rightarrow 0$. 
\begin{lem}
	\begin{flalign*}
		\underset{\epsilon \rightarrow 0}{\lim \inf} \lambda_{\epsilon}^{\phi} >0.
	\end{flalign*}
\end{lem}

\begin{proof}
	It is not diﬃcult to ﬁnd that $u_{\epsilon}$ satisﬁes the Euler–Lagrange equation
	\begin{equation} 
		\label{eq:4.1}
		\begin{cases}
			 -\operatorname{div} \left( |\nabla u_{\epsilon}|^{N-2} \nabla u_{\epsilon} \right) = \frac{1}{ \lambda_{\epsilon}^{\phi} } |x|^{N \epsilon} |u_{\epsilon}|^{ \frac{1}{N-1} } \Phi^{\prime} \left( \alpha_N \left( 1+ \epsilon \right) |u_{\epsilon}|^{ \frac{N}{N-1}} \right),\\
			 u_{\epsilon} >0, ~ \mathbb{B},\\
			 u_{\epsilon} =0 , ~ \partial \mathbb{B},\\
			 \lambda_{\epsilon}^{\phi} = \int_{\mathbb{B}} |x|^{N \epsilon} |u_{\epsilon}|^{ \frac{N}{N-1} } \Phi^{\prime} \left( \alpha_N \left( 1 + \epsilon \right) |u_{\epsilon}|^{ \frac{N}{N-1} } \right) dx.
		\end{cases}
	\end{equation}
    From the boundedness of $\{u_{\epsilon} \}_{\epsilon}$ in $W_0^{1,N} \left( \mathbb{B} \right)$, we can assume that
    \begin{flalign} \label{eq: 4.2}
    	\begin{split}
    		& u_{\epsilon} \rightharpoonup u, ~ W_0^{1,N} \left( \mathbb{B} \right);\\
    		& u_{\epsilon} \rightarrow u, ~ L^p \left( \mathbb{B} \right), ~ \forall p >1;\\
    		& u_{\epsilon} \rightarrow u, ~ \text{a.e.} ~ x \in \mathbb{B}.\\
    	\end{split}
    \end{flalign}
    By the Lebesgue dominated convergence theorem, we obtain
    \begin{flalign*}
    	\int_{\mathbb{B}} \Phi^{\prime} \left( \alpha_N |u|^{ \frac{N}{N-1} } \right) dx & = \underset{\epsilon \rightarrow 0}{\lim} \int_{\mathbb{B}} |x|^{N \epsilon} \Phi^{\prime} \left( \alpha_N\left( 1+ \epsilon \right) |u|^{ \frac{N}{N-1} } \right) dx\\
    	& \leq \underset{\epsilon \rightarrow 0}{\lim \inf} \int_{\mathbb{B}} |x|^{N \epsilon} \Phi^{\prime} \left( \alpha_N\left( 1+ \epsilon \right) |u_{\epsilon}|^{ \frac{N}{N-1} } \right) dx.
    \end{flalign*}
    We show that $\underset{\epsilon >0}{\inf} \lambda_{\epsilon}^{\phi} >0$; which follows from \cite{7}.
\end{proof}

\begin{lem}
	Let $u$ be the limit of $\{ u_{\epsilon} \}_{\epsilon}$ in \ref{eq:4.1} and \ref{eq: 4.2}. Then $u = 0$ and 
	\begin{flalign*}
		|\nabla u_{\epsilon}|^N \rightharpoonup \delta_0
	\end{flalign*}
    weakly in the sense of measures, where $\delta_0$ stands for the Dirac measure
	centered at the origin.
\end{lem}

\begin{proof}
	We assume that
	\begin{flalign*}
		f_{\epsilon} = \frac{1}{\lambda_{\epsilon}^{\phi}} |x|^{N \epsilon} \Phi^{\prime} \left( \alpha_N \left( 1+ \epsilon \right) |u_{\epsilon}|^{ \frac{N}{N-1}} \right) dx.
	\end{flalign*}
    We will show that
    \begin{flalign*}
    	\int_{\mathbb{B}} f_{\epsilon}^p dx \leq C < \infty,
    \end{flalign*}
    for some $p >1$ and all $\epsilon > 0$. Indeed, by the H\"{o}lder inequality, we get
    \begin{flalign*}
    	\int_{\mathbb{B}} f_{\epsilon}^p dx & \leq \frac{1}{ \lambda_{\epsilon}^{\phi p} } \int_{ \mathbb{B} } |x|^{N p \epsilon} \Phi^{\prime p} \left( \alpha_N \left( 1+ \epsilon \right) |u_{\epsilon}|^{ \frac{N}{N-1}} \right) dx\\
    	& \leq \frac{1}{ \lambda_{\epsilon}^{\phi p} } \int_{ \mathbb{B} } |x|^{N p \epsilon} \Phi^{\prime} \left( \alpha_N \left( 1+ \epsilon \right) p |u_{\epsilon}|^{ \frac{N}{N-1}} \right) dx\\
    	& \leq \frac{1}{ \lambda_{\epsilon}^{\phi p} } \int_{ \mathbb{B} } \Phi^{\prime} \left( \alpha_N \left( 1+ \epsilon \right) p |u_{\epsilon}|^{ \frac{N}{N-1}} \right) dx < \infty, 
    \end{flalign*}
    with $p < \frac{1}{ \left( 1 - ||\nabla u||^N_{L^N} \right)^{ \frac{1}{N-1} } }$ and by Lemma \ref{lem: 2.4}. Then we have $|u_{\epsilon}|^{ \frac{1}{N-1} } f_{\epsilon} \in L^p \left( \mathbb{B} \right)$. Then, we have $||u_{\epsilon}||_{L^{\infty}} < \infty$ by elliptic estimate theory. It is impossible since $c_{\epsilon} \rightarrow + \infty$ as $\epsilon \rightarrow 0$. Hence, $u = 0$.
    
    In addition, we will show $|\nabla u_{\epsilon}|^N \rightharpoonup \delta_0$ as $\epsilon \rightarrow 0$. If not, there exist constants $\kappa \in (0,1)$ such that
    \begin{flalign*}
    	\underset{\epsilon \rightarrow 0}{\lim \sup} ||\nabla u_{\epsilon}||^N_{L^N} = \kappa <1.
    \end{flalign*}
    Consider the equation
    \begin{flalign*}
    	& \int_{\mathbb{B}} \left( |x|^{N \epsilon} |u_{\epsilon}|^{ \frac{1}{N-1} } \Phi^{\prime} \left( \alpha_N \left( 1+ \epsilon \right) |u_{\epsilon}|^{ \frac{N}{N-1} } \right) \right)^ p dx\\
    	& \leq \int_{\mathbb{B}} |u_{\epsilon}|^{ \frac{p}{N-1} } \Phi^{\prime p} \left( \alpha_N \left( 1+ \epsilon \right) |u_{\epsilon}|^{ \frac{N}{N-1} } \right)  dx\\
    	&  \leq \left( \int_{ \mathbb{B}} |u_{\epsilon}|^{ \frac{p_1 \cdot p}{N-1} } dx \right)^{ \frac{1}{p_1} } \left( \Phi^{ \prime p_2 \cdot p} \left( \alpha_N \left( 1+ \epsilon \right) |u_{\epsilon}|^{ \frac{N}{N-1} } \right) \right)^{ \frac{1}{p_2} } \\
    	& \leq \left( \int_{ \mathbb{B}} |u_{\epsilon}|^{ \frac{p_1 \cdot p}{N-1} } dx \right)^{ \frac{1}{p_1} } \left( \Phi^{ \prime } \left( \alpha_N \left( 1+ \epsilon \right) p_2 \cdot p |u_{\epsilon}|^{ \frac{N}{N-1} } \right) \right)^{ \frac{1}{p_2} }.
    \end{flalign*}
    By the concentration-compactness principle, we have that $ \frac{1}{ \lambda_{\epsilon}^{\phi} } |x|^{N \epsilon} |u_{\epsilon}|^{ \frac{1}{N-1}  } \Phi^{\prime} $ are bounded in some $L^{p} \left( \mathbb{B} \right)$ spaces. By using the elliptic estimate 
    again, we have that $|| u_{\epsilon}||_{ L^{\infty} } < \infty$. It is impossible that $c_{\epsilon} \rightarrow + \infty$. Then we know that $|\nabla u_{\epsilon}|^N \rightharpoonup \delta_0$. 
\end{proof}

Considering the constants $r_{\epsilon} = \lambda_{\epsilon}^{ \phi \frac{1}{N} } c_{\epsilon}^{- \frac{1}{N-1}} \exp \left\lbrace - \frac{\alpha_N \left( 1+ \epsilon \right)}{N} c_{\epsilon}^{ \frac{N}{N-1} } \right\rbrace $, we have the following lemma,
\begin{lem} 
	If $c_{\epsilon} \rightarrow + \infty$, then for any $0 < \varrho < \alpha_N$,
	\begin{flalign*}
		\underset{\epsilon \rightarrow 0}{\lim} r_{\epsilon}^N \Phi \left( \varrho \left( 1+ \epsilon \right) c_{\epsilon}^{ \frac{N}{N-1} } \right) = 0.
	\end{flalign*}
    Moreover, we have
	\begin{flalign*}
		\underset{\epsilon \rightarrow 0}{\lim} r_{\epsilon} =0.
	\end{flalign*}
\end{lem}

\begin{proof}
	By direct calculation, we have
	\begin{flalign*}
		r_{\epsilon}^N \Phi \left( \varrho \left( 1+ \epsilon \right) c_{\epsilon}^{ \frac{N}{N-1} } \right) & = \lambda_{\epsilon}^{\phi} c_{\epsilon}^{- \frac{N}{N-1}} e^{ - \alpha_N \left( 1+ \epsilon \right) c_{\epsilon}^{ \frac{N}{N-1} } } \Phi \left( \varrho \left( 1+ \epsilon \right) c_{\epsilon}^{ \frac{N}{N-1} } \right)\\
		& = \frac{ \lambda_{\epsilon}^{\phi} }{ c_{\epsilon}^{ \frac{N}{N-1}} e^{ \alpha_N \left( 1+ \epsilon \right) c_{\epsilon}^{ \frac{N}{N-1} } } } \times\\
		& \left( e^{ \varrho \left( 1+ \epsilon \right) |c_{\epsilon}|^{ \frac{N}{N-1} } } - \sum_{n=0}^{N-2} \frac{ \varrho^n \left( 1+ \epsilon \right)^n |c_{\epsilon}|^{\frac{n N}{N-1}}  }{n !} \right).
	\end{flalign*}
    We also have that
    \begin{flalign*}
    	\lambda_{\epsilon}^{\phi} = \int_{\mathbb{B}} |x|^{N \epsilon} |u_{\epsilon}|^{ \frac{N}{N-1} } \Phi^{ \prime} \left( \alpha_N \left( 1+ \epsilon \right) |u_{\epsilon}|^{ \frac{N}{N-1} } \right) dx.
    \end{flalign*}
    Then we have the following estimates
    \begin{flalign*}
    	\lambda_{\epsilon}^{\phi} & \leq |c_{\epsilon}|^{ \frac{N}{N-1} } \Phi^{\prime} \left( \frac{\alpha_N \left( 1+ \epsilon \right) |c_{\epsilon}|^{ \frac{N}{N-1} }}{2} \right) \int_{\mathbb{B}} |x|^{N \epsilon} \Phi^{ \prime} \left( \frac{\alpha_N \left( 1+ \epsilon \right) |u_{\epsilon}|^{ \frac{N}{N-1} }}{2} \right) dx\\
    	& \leq C |c_{\epsilon}|^{ \frac{N}{N-1} } \Phi^{\prime} \left( \frac{\alpha_N \left( 1+ \epsilon \right) |c_{\epsilon}|^{ \frac{N}{N-1} }}{2} \right).
    \end{flalign*}
    We also have
    \begin{flalign*}
    	r_{\epsilon}^N c_{\epsilon}^{ \frac{N}{N-1} } & = \lambda_{\epsilon}^{\phi} e^{ - \alpha_N \left( 1+ \epsilon \right) c_{\epsilon}^{ \frac{N}{N-1} } }\\
    	& \leq C c_{\epsilon}^{ \frac{N}{N-1} } e^{ -\frac{ \alpha_N \left( 1+ \epsilon \right) c_{\epsilon}^{ \frac{N}{N-1} } }{2} } .
    \end{flalign*}
    Then the proof is completed.
\end{proof}

\begin{lem}
	We set 
	\begin{flalign*}
		& \psi_{\epsilon} = c_{\epsilon}^{-1} u_{\epsilon} \left(r_{\epsilon}^{ \frac{1}{1 + \epsilon} } x \right),\\
		& \phi_{\epsilon} = \gamma_{\epsilon} c_{\epsilon}^{ \frac{1}{N-1} } \left( u_{\epsilon} \left(r_{\epsilon}^{ \frac{1}{1 + \epsilon} } x \right) - c_{\epsilon} \right).
	\end{flalign*}
    Then 
    \begin{flalign*}
    	& \psi_{\epsilon} \rightarrow 1, ~ C^1_{loc} \left( \mathbb{R}^N \right),\\
    	& \phi_{\epsilon} \rightarrow \phi_0, ~ C^1_{loc} \left( \mathbb{R}^N \right),
    \end{flalign*}
    where $r_{\epsilon} = \frac{\lambda_{\epsilon}^{ \frac{1}{N} } }{ c_{\epsilon}^{ \frac{1}{N-1} } e^{ \frac{\alpha_N}{N} \left( 1+ \epsilon \right) |c_{\epsilon}|^{ \frac{N}{N-1} } } }$ and $\gamma_{\epsilon} = \frac{N}{N-1} \alpha_N$.
\end{lem}

\begin{proof}
	From direct calculation,
    \begin{flalign*}
    	- \operatorname{div} \left( |\nabla \psi_{\epsilon} |^{N-2} \nabla \psi_{\epsilon} \right) & = c_{\epsilon}^{-N+ \frac{1}{N-1} } r_{\epsilon}^{ N} \frac{1}{ \lambda_{\epsilon}^{\phi} } |x|^{N \epsilon} |\psi_{\epsilon}|^{ \frac{1}{N-1} } \Phi^{\prime} \left( \alpha_N \left( 1+ \epsilon \right) |\tilde{ u }_{\epsilon}|^{ \frac{N}{N-1} } \right)\\
    	& = c_{\epsilon}^{ - \left( N +1 \right) } |x|^{N \epsilon} |\psi_{\epsilon} |^{ \frac{1}{N-1} } e^{ - \alpha_N \left( 1 + \epsilon \right) |c_{\epsilon}|^{ \frac{N}{N-1} } } \Phi^{ \prime } \left( \alpha_N \left( 1 + \epsilon \right) |\tilde{u}_{\epsilon}|^{ \frac{N}{N-1} } \right).
    \end{flalign*}
    \begin{flalign*}
    	-\operatorname{div}  \left( |\nabla \phi_{\epsilon} |^{N-2} \nabla \phi_{\epsilon} \right)&  = - \left( \gamma_{\epsilon} c_{\epsilon}^{ \frac{1}{N-1} } r_{\epsilon}^{ \frac{1}{1+ \epsilon} } \right)^{N-1} r_{\epsilon}^{ \frac{1}{1+ \epsilon} } \operatorname{div}  \left( |\nabla \tilde{u}_{\epsilon} |^{N-2} \nabla \tilde{u}_{\epsilon} \right)\\
    	& = e^{ - \alpha_N \left( 1+ \epsilon \right) |c_{\epsilon}|^{ \frac{N}{N-1} } } \gamma_{\epsilon}^{N-1} |x|^{N \epsilon} | \psi_{\epsilon}|^{ \frac{1}{N-1} } \Phi^{\prime} \left( \alpha_N \left( 1+ \epsilon \right) |\tilde{ u }_{\epsilon}|^{ \frac{N}{N-1} } \right).
    \end{flalign*}
    From \cite{5} and \cite{6}, we have  
    \begin{flalign*}
    	-\operatorname{div} \left( |\nabla \phi_{\epsilon}|^{N-2} \nabla \phi_{\epsilon} \right) & = \gamma_{\epsilon}^{N-1} |x|^{N \epsilon} |\psi_{\epsilon}|^{ \frac{1}{N-1} } \exp \left\lbrace \alpha_N \left( 1+ \epsilon \right) \left( |\tilde{ u }_{\epsilon} |^{ \frac{N}{N-1} } - |c_{\epsilon}|^{ \frac{N}{N-1} }\right) \right\rbrace \\
    	& + O \left( r_{\epsilon}^N c_{\epsilon}^N \right),
    \end{flalign*}
    \begin{flalign*}
    	& \underset{ \mathbb{B}_R }{osc} \phi_{\epsilon} \leq C(R),\\
    	& ||\phi_{\epsilon}||_{ C^{1, \alpha} } < C(R).
    \end{flalign*}
    Therefor, we have
    \begin{flalign*}
    	& \phi_{\epsilon} \rightarrow \phi_0, ~ C^{1,\alpha}_{loc} \left( \mathbb{B}_R \right);\\
    	& \tilde{u}_{\epsilon} - c_{\epsilon} \rightarrow 0, ~ C^{1,\alpha}_{loc} \left( \mathbb{B}_R \right).
    \end{flalign*}
    Then we have results as follows, $\alpha_N \left( 1+ \epsilon \right) \left( |\tilde{ u}_{\epsilon}  |^{ \frac{N}{N-1} } - |c_{\epsilon}|^{ \frac{N}{N-1} }\right) \rightarrow \phi $ as $\epsilon \rightarrow 0$ in $C^0_{loc} \left( \mathbb{B}_R\right)$,
    \begin{flalign*}
    	- \operatorname{div} \left( |\nabla \phi|^{N-2} \nabla \phi \right) = \gamma^{N-1} e^{\phi},
    \end{flalign*}
    where $\gamma = \gamma_{\epsilon}$. Since $\phi$ is radially symmetric and decreasing, it is easy to see that the equation has only one solution $\phi_0$. Moreover, we have 
    \begin{flalign*}
    	& \phi_0 (0) = 0 = \max \phi_0 (x).\\
    	& \phi_0 (x) = - N \ln \left( 1 + \left( \frac{ \omega_{N-1} }{N } \right)^{ \frac{1}{N-1} } |x|^{ \frac{N}{N-1} } \right), ~ \int_{ \mathbb{R}^N } e^{\phi_0} dx =1.\\
    \end{flalign*}
\end{proof}

\begin{lem}
	For any $0 < \varrho <1$, we have 
	\begin{flalign*}
		\underset{\epsilon \rightarrow 0}{\lim \sup} \int_{\mathbb{B}} |\nabla u_{\epsilon, \varrho}|^N dx \leq \varrho, 
	\end{flalign*}
    where $u_{\epsilon, \varrho} := \min \{u_{\epsilon}, \varrho c_{\epsilon} \}$.
\end{lem}

\begin{proof}
	We consider the equation
	\begin{flalign*}
		& \int_{\mathbb{B}} |\nabla \left(u_{\epsilon}- \varrho c_{\epsilon} \right)^+ |^N dx \\
		& = \frac{1}{\lambda_{\epsilon}^{\phi} } \int_{ \mathbb{B} } |x|^{N \epsilon} \left(u_{\epsilon}- \varrho c_{\epsilon} \right)^+ |u_{\epsilon}|^{ \frac{1}{N-1} } e^{ \alpha_N \left( 1+ \epsilon \right) |u_{\epsilon}|^{ \frac{N}{N-1} } } dx + o_{\epsilon} (1)\\
		& \geq \frac{1}{\lambda_{\epsilon}^{\phi} } \int_{ \mathbb{B}_{R r^{ \frac{1}{1 + \epsilon} }_{\epsilon} } } |x|^{N \epsilon} \left(u_{\epsilon}- \varrho c_{\epsilon} \right)^+ |u_{\epsilon}|^{ \frac{1}{N-1} } e^{ \alpha_N \left( 1+ \epsilon \right) |u_{\epsilon}|^{ \frac{N}{N-1} } } dx + o_{\epsilon} (1)\\
		& = c_{\epsilon}^{ - \frac{N}{N-1} } \exp \left\lbrace - \alpha_N \left( 1+ \epsilon \right) c_{\epsilon}^{ \frac{N}{N-1} }  \right\rbrace \times\\
		& \int_{\mathbb{B}_R} |y|^{N \epsilon} \left( \tilde{u}_{\epsilon} - \varrho c_{\epsilon} \right)^+ |\tilde{u}_{\epsilon}|^{ \frac{1}{N-1} } e^{ \alpha_N \left( 1+ \epsilon \right) |\tilde{u}_{\epsilon}|^{ \frac{N}{N-1} } } dy + o_{\epsilon} (1).
	\end{flalign*}
    We also have
    \begin{flalign*}
    	|\tilde{u}_{\epsilon}|^{ \frac{N}{N-1} } & = c_{\epsilon}^{ \frac{N}{N-1} } \left( 1+ \frac{ |\tilde{u}_{\epsilon}|  - c_{\epsilon} }{ c_{\epsilon} } \right)^{ \frac{N}{N-1} }\\
    	& = c_{\epsilon}^{ \frac{N}{N-1} } \left( 1 + \frac{N}{N-1} \frac{ |\tilde{u}_{\epsilon}|  - c_{\epsilon} }{ c_{\epsilon} } + O \left(  \frac{1}{ c_{\epsilon}^2 } \right) \right).
    \end{flalign*}
    Then 
    \begin{flalign*}
    	& \int_{\mathbb{B}_R}|y|^{N \epsilon} \left( \tilde{u}_{\epsilon} - \varrho c_{\epsilon} \right)^+ |\tilde{u}_{\epsilon}|^{ \frac{1}{N-1} } e^{ \alpha_N \left( 1+ \epsilon \right) |\tilde{u}_{\epsilon}|^{ \frac{N}{N-1} } } dy\\
    	& = \int_{\mathbb{B}_R} |y|^{N 
    	\epsilon} \left( \frac{ \tilde{u}_{\epsilon} - \varrho c_{\epsilon} }{ c_{\epsilon} } \right)^+ \left| \frac{\tilde{u}_{\epsilon}}{c_{\epsilon}} \right|^{ \frac{1}{N-1} } e^{ \alpha_N \left( 1 + \epsilon \right) \left( |\tilde{u}_{\epsilon}|^{ \frac{N}{N-1} } - c_{\epsilon}^{ \frac{N}{N-1} } \right) } dy. 
    \end{flalign*}
    Then we have 
    \begin{flalign*}
    	& \int_{\mathbb{B}_R} |y|^{N 
    		\epsilon} \left( \frac{ \tilde{u}_{\epsilon} - \varrho c_{\epsilon} }{ c_{\epsilon} } \right)^+ \left| \frac{\tilde{u}_{\epsilon}}{c_{\epsilon}} \right|^{ \frac{1}{N-1} } e^{ \alpha_N \left( 1 + \epsilon \right) \left( |\tilde{u}_{\epsilon}|^{ \frac{N}{N-1} } - c_{\epsilon}^{ \frac{N}{N-1} } \right) } dy\\
    	& = \int_{\mathbb{B}_R} |y|^{N 
    		\epsilon} \left( \frac{ \tilde{u}_{\epsilon} - \varrho c_{\epsilon} }{ c_{\epsilon} } \right)^+ \left| \frac{\tilde{u}_{\epsilon}}{c_{\epsilon}} \right|^{ \frac{1}{N-1} } e^{ \alpha_N \left( 1+\epsilon \right) \frac{N \cdot c_{\epsilon}^{ \frac{1}{N-1} } }{N-1} \left( |\tilde{u}_{\epsilon}| - c_{\epsilon} \right) + o(1)} dy \\
    	&= \int_{\mathbb{B}_R} |y|^{N 
    		\epsilon} \left( \frac{ \tilde{u}_{\epsilon} - \varrho c_{\epsilon} }{ c_{\epsilon} } \right)^+ \left| \frac{\tilde{u}_{\epsilon}}{c_{\epsilon}} \right|^{ \frac{1}{N-1} } e^{\phi_{\epsilon} + o(1)} dy .
    \end{flalign*}
    Therefor, we have
    \begin{flalign*}
    	\underset{\epsilon \rightarrow 0}{\lim \inf} \int_{\mathbb{B}} |\nabla \left( u_{\epsilon} - \varrho c_{\epsilon} \right)^+|^N dx \geq \left( 1- \varrho \right) \int_{\mathbb{B}_R} e^{\phi } dx.
    \end{flalign*}
    Letting $R \rightarrow + \infty$, we have 
    \begin{flalign*}
    	\underset{\epsilon \rightarrow 0}{\lim \inf} \int_{\mathbb{B}} |\nabla \left( u_{\epsilon} - \varrho c_{\epsilon}  \right)^+|^N dx \geq 1- \varrho.
    \end{flalign*}
    Then 
    \begin{flalign*}
        \int_{\mathbb{B}} |\nabla u_{\epsilon, \varrho} |^N dx & = 1 - \int_{\mathbb{B}} |\nabla \left( u_{\epsilon} - \varrho c_{\epsilon} \right)^+|^N dx\\
        & \leq 1- \left( 1- \varrho \right) + o(1).
    \end{flalign*}
    Hence, we obtain this lemma.
\end{proof}

\begin{lem} \label{lem: 4.6}
	\begin{flalign*}
		& \underset{\epsilon \rightarrow 0}{\lim} \int_{\mathbb{B}} |x|^{N \epsilon} \Phi \left( \alpha_N \left( 1+ \epsilon \right) |u_{\epsilon}|^{ \frac{N}{N-1} } \right) dx\\
		&  \leq \underset{\epsilon \rightarrow 0}{\lim \sup} \frac{ \lambda_{\epsilon} }{ c_{\epsilon}^{ \frac{N}{N-1} } }.
	\end{flalign*}
\end{lem}

\begin{proof}
	We consider the equation
	\begin{flalign*}
		& \int_{\mathbb{B}} |x|^{N \epsilon} \Phi \left( \alpha_N\left(1+ \epsilon \right) |u_{\epsilon}|^{ \frac{N}{N-1} } \right) dx\\
		& = \int_{ u_{\epsilon} \leq \varrho c_{\epsilon} }  |x|^{N \epsilon} \Phi \left( \alpha_N\left(1+ \epsilon \right) |u_{\epsilon}|^{ \frac{N}{N-1} } \right) dx\\
		& + \int_{ u_{\epsilon} > \varrho c_{\epsilon} }  |x|^{N \epsilon} \Phi \left( \alpha_N\left(1+ \epsilon \right) |u_{\epsilon}|^{ \frac{N}{N-1} } \right) dx \\
		& := I + II.
	\end{flalign*}
    Then,
    \begin{flalign*}
    	I & = \int_{ u_{\epsilon} \leq \varrho c_{\epsilon}} |x|^{N \epsilon} \Phi \left( \alpha_N \left( 1 + \epsilon \right) |u_{\epsilon}|^{ \frac{N}{N-1} } \right) dx\\
    	& \leq \int_{ \mathbb{B} } |x|^{N \epsilon} \Phi \left( \alpha_N \left( 1 + \epsilon \right) |u_{\epsilon, \varrho}|^{ \frac{N}{N-1} } \right) dx.
    \end{flalign*}
    Applying the concentration-compactness principle and definition of $\Phi$, we have
    \begin{flalign*}
    	& \int_{\mathbb{B}} |x|^{N \epsilon} \Phi \left( \alpha_N \left( 1 + \epsilon \right) |u_{\epsilon, \varrho}|^{ \frac{N}{N-1} } \right) dx\\
    	& = \int_{\mathbb{B} \setminus \mathbb{B}_R }  |x|^{N \epsilon} \Phi \left( \alpha_N \left( 1 + \epsilon \right) |u_{\epsilon, \varrho}|^{ \frac{N}{N-1} } \right) dx\\
    	& + \int_{\mathbb{B}_R} |x|^{N \epsilon} \Phi \left( \alpha_N \left( 1 + \epsilon \right) |u_{\epsilon, \varrho}|^{ \frac{N}{N-1} } \right) dx, 
    \end{flalign*}
    and
    \begin{flalign*}
    	\underset{ \epsilon >0 }{\sup}  \int_{\mathbb{B}_R} |x|^{N \epsilon} \Phi \left( \alpha_N \left( 1 + \epsilon \right) |u_{\epsilon, \varrho}|^{ \frac{N}{N-1} } \right) dx < \infty.
    \end{flalign*}
    Also, from the Poincar\'{e} inequality,
    \begin{flalign*}
    	\int_{\mathbb{B} \setminus \mathbb{B}_R} |x|^{N \epsilon} \Phi \left( \alpha_N \left( 1 + \epsilon \right) |u_{\epsilon, \varrho}|^{ \frac{N}{N-1} } \right) dx & \leq C \int_{\mathbb{B} \setminus \mathbb{B}_R} |u_{\epsilon}|^N dx\\
    	\leq ||\nabla u_{\epsilon}||_{L^N}^N \rightarrow 0,
    \end{flalign*}
    as $\epsilon \rightarrow 0$.
    On the other hand, we have
    \begin{flalign*}
    	& \int_{\mathbb{B} \cap \{ u_{\epsilon} > \varrho c_{\epsilon} \} } |x|^{N \epsilon} \Phi \left( \alpha_N \left( 1+ \epsilon \right) |u_{\epsilon}|^{ \frac{N}{N-1} } \right) dx \\
    	& \leq \int_{\mathbb{B}} |x|^{N \epsilon} \frac{ |u_{\epsilon}|^{ \frac{N}{N-1} } }{ \left( \varrho c_{\epsilon} \right)^{ \frac{N}{N-1} } }  \Phi \left( \alpha_N \left( 1+ \epsilon \right) |u_{\epsilon}|^{ \frac{N}{N-1} } \right) dx\\
    	& = \frac{ \lambda_{\epsilon} }{ \lambda_{\epsilon} } \int_{\mathbb{B}} |x|^{N \epsilon} \frac{ |u_{\epsilon}|^{ \frac{N}{N-1} } }{ \left( \varrho c_{\epsilon} \right)^{ \frac{N}{N-1} } }  \Phi^{\prime} \left( \alpha_N \left( 1+ \epsilon \right) |u_{\epsilon}|^{ \frac{N}{N-1} } \right) dx.
    \end{flalign*}
    Then,
    \begin{flalign*}
    	& \underset{\epsilon \rightarrow 0}{\lim \sup} \frac{ \lambda_{\epsilon} }{ \left( \varrho c_{\epsilon} \right)^{ \frac{N}{N-1} } } \int_{\mathbb{B}_R } \frac{1}{ \lambda_{\epsilon} } |x|^{ N \epsilon } |u_{\epsilon}|^{ \frac{N}{N-1} } \Phi^{\prime} \left( \alpha_N \left( 1+ \epsilon \right) |u_{\epsilon}|^{ \frac{N}{N-1} } \right) dx + C \epsilon\\
    	& = \frac{ \lambda_{\epsilon} }{ \left( \varrho c_{\epsilon} \right)^{ \frac{N}{N-1} } },
    \end{flalign*}
    as $\varrho \rightarrow 1$ and $\epsilon \rightarrow 0$, the result follows.
\end{proof}

\begin{lem}
	For any $0 < \gamma < \frac{N}{N-1}$, we have
	\begin{flalign*}
		\underset{\epsilon \rightarrow 0}{\lim} \frac{\lambda_{\epsilon} }{ c_{\epsilon}^{\gamma} } = + \infty.
	\end{flalign*}
    Also,
    \begin{flalign*}
    	\underset{\epsilon >0 }{\sup} \frac{ c_{\epsilon}^{ \frac{N}{n-1} } }{ \lambda_{\epsilon} } < \infty.
    \end{flalign*}
\end{lem}

\begin{proof}
	If $\underset{\epsilon > 0}{\sup} \frac{c_{\epsilon}^{ \frac{N}{N-1} } }{ \lambda_{\epsilon } }= +\infty$, then we have
		\begin{flalign*}
			\underset{\epsilon \rightarrow 0}{\lim \sup} \frac{ \lambda_{\epsilon }}{c_{\epsilon}^{ \frac{N}{N-1} } } =0.
		\end{flalign*}
	However, Lemma \ref{lem: 4.6} means
		\begin{flalign*}
			\underset{\epsilon \rightarrow 0}{\lim} \int_{\mathbb{B}} |x|^{N \epsilon} \Phi \left( \alpha_N \left( 1+ \epsilon \right)  |u_{\epsilon }|^{ \frac{N}{n-1} } \right) dx = 0,
		\end{flalign*}
	which is impossible.
	On the other hand, if 
	\begin{flalign*}
		\underset{\epsilon \rightarrow 0}{\lim} \frac{\lambda_{\epsilon} }{ c_{\epsilon}^{\gamma} } < +\infty.
	\end{flalign*}
    Then 
    \begin{flalign*}
    	\underset{\epsilon \rightarrow 0}{\lim} \frac{\lambda_{\epsilon} }{ c_{\epsilon}^{\gamma} } \frac{1}{ c_{\epsilon}^{ \frac{N}{N-1} - \gamma}  }=0.
    \end{flalign*}
    Then the lemma is completed.
\end{proof}

\begin{lem}
	We have $c_{\epsilon} |x|^{N \epsilon} \frac{ |u_{\epsilon} |^{ \frac{1}{N-1} } }{ \lambda_{\epsilon}^{\phi} } \Phi^{\prime} \left( \alpha_N \left( 1+ \epsilon \right) |u_{\epsilon}|^{ \frac{N}{N-1} } \right) \rightharpoonup \delta_0$; that is, for any $\varphi \in \mathcal{D} \left( \mathbb{B} \right)$, we have
	\begin{flalign*}
		\underset{\epsilon \rightarrow 0}{\lim} \int_{ \mathbb{R}^N } \varphi \cdot c_{\epsilon} |x|^{N \epsilon} \frac{ |u_{\epsilon} |^{ \frac{1}{N-1} } }{ \lambda_{\epsilon}^{\phi} } \Phi^{\prime} \left( \alpha_N \left( 1+ \epsilon \right) |u_{\epsilon}|^{ \frac{N}{N-1} } \right) dx = \varphi (0).
	\end{flalign*}
\end{lem}

\begin{proof}
	We consider the equation
	\begin{flalign*}
		& \int_{ \mathbb{B} } \varphi \cdot c_{\epsilon} |x|^{N \epsilon} \frac{ |u_{\epsilon} |^{ \frac{1}{N-1} } }{ \lambda_{\epsilon}^{\phi} } \Phi^{\prime} \left( \alpha_N \left( 1+ \epsilon \right) |u_{\epsilon}|^{ \frac{N}{N-1} } \right) dx\\
		& = \int_{ \{ u_{\epsilon} > \varrho c_{\epsilon}\} \cap \mathbb{B}_{ R r_{\epsilon}^{ \frac{1}{ 1+ \epsilon } } } } + \int_{ \{ u_{\epsilon} > \varrho c_{\epsilon}\} \setminus \mathbb{B}_{ R r_{\epsilon}^{ \frac{1}{ 1+ \epsilon } } } } + \int_{ \{ u_{\epsilon} < \varrho c_{\epsilon}\} }\\
		& := I + II + III.
	\end{flalign*}
    From then on,
    \begin{flalign*}
    	& \int_{  \{ u_{\epsilon} > \varrho c_{\epsilon}\} \setminus \mathbb{B}_{ R r_{\epsilon}^{ \frac{1}{ 1+ \epsilon } } }  } \varphi \cdot |x|^{N \epsilon} c_{\epsilon} \frac{ |u_{\epsilon} |^{ \frac{1}{N-1} } }{ \lambda_{\epsilon}^{\phi} } \Phi^{\prime} \left( \alpha_N \left( 1+ \epsilon \right) |u_{\epsilon}|^{ \frac{N}{N-1} } \right) dx\\
    	& \leq ||\varphi||_{ L^{\infty} } \frac{1}{ \varrho } \int_{ \mathbb{B} \setminus \mathbb{B}_{ R r_{\epsilon}^{ \frac{1}{ 1+ \epsilon } } } } |x|^{N \epsilon} \frac{ |u_{\epsilon} |^{ \frac{N}{N-1} } }{ \lambda_{\epsilon}^{\phi} } \Phi^{\prime} \left( \alpha_N \left( 1+ \epsilon \right) |u_{\epsilon}|^{ \frac{N}{N-1} } \right) dx\\
    	& \leq ||\varphi||_{ L^{\infty} } \frac{1}{ \varrho } \left( 1 - \left( 1- \varrho \right) \int_{\mathbb{B}_R} e^{\phi_{\epsilon} + o(1) } dx \right).
    \end{flalign*}
    Also, from the H\"{o}lder inequality, we have
    \begin{flalign*}
    	& \int_{  \{ u_{\epsilon} < \varrho c_{\epsilon}\} } \varphi \cdot |x|^{N \epsilon} c_{\epsilon} \frac{ |u_{\epsilon} |^{ \frac{1}{N-1} } }{ \lambda_{\epsilon}^{\phi} } \Phi^{\prime} \left( \alpha_N \left( 1+ \epsilon \right) |u_{\epsilon}|^{ \frac{N}{N-1} } \right) dx\\
    	& \leq || \varphi||_{ L^{\infty} } \frac{ c_{\epsilon} }{ \lambda_{\epsilon}^{\phi} } \int_{ \mathbb{B} }  |x|^{N \epsilon}    |u_{\epsilon, \varrho} |^{ \frac{1}{N-1} }  e^{ \alpha_N \left( 1+ \epsilon \right) |u_{\epsilon, \varrho}|^{ \frac{N}{N-1} } } dx\\
    	& \leq || \varphi||_{ L^{\infty} } \frac{ c_{\epsilon} }{ \lambda_{\epsilon}^{\phi} } \varrho \rightarrow 0,
    \end{flalign*}
    as $\epsilon \rightarrow 0$. Finally, we consider as $r_{\epsilon}^{ \frac{1}{ 1+ \epsilon } }$ is small enough, 
    \begin{flalign*}
    	& \int_{ \{ u_{\epsilon} > \varrho c_{\epsilon}\} \setminus \mathbb{B}_{ R r_{\epsilon}^{ \frac{1}{ 1+ \epsilon } } } } \varphi \cdot |x|^{N \epsilon} \frac{ |u_{\epsilon} |^{ \frac{1}{N-1} } }{ \lambda_{\epsilon}^{\phi} } \Phi^{\prime} \left( \alpha_N \left( 1+ \epsilon \right) |u_{\epsilon}|^{ \frac{N}{N-1} } \right) dx\\
    	& = \int_{ \mathbb{B}_{ R } } \tilde{\varphi} \cdot |y|^{N \epsilon} \frac{ |\tilde{u}_{\epsilon} |^{ \frac{1}{N-1} } }{ \lambda_{\epsilon}^{\phi} } e^{ \alpha_N \left( 1+ \epsilon \right) \left(|\tilde{u}_{\epsilon}|^{ \frac{N}{N-1} } -  c_{\epsilon}^{ \frac{N}{N-1} }\right) }  dy +o(1)\\
    	& = \varphi (0) \int_{\mathbb{B}_R} e^{\phi_0} dx + o(1).
    \end{flalign*}
\end{proof}

\begin{lem}
	\begin{flalign*}
		& \underset{\epsilon \rightarrow 0}{\lim \sup} \int_{\mathbb{B}} |x|^{N \epsilon} \Phi \left( \alpha_N \left( 1+ \epsilon \right) |u_{\epsilon}|^{ \frac{N}{N-1} } \right) dx \\
		& = |\mathbb{B}| + \underset{R \rightarrow + \infty}{\lim }\underset{\epsilon \rightarrow 0}{\lim \sup} \int_{\mathbb{B}_{R r_{\epsilon}^{ \frac{1}{1+ \epsilon} }} } |x|^{N \epsilon} \Phi \left( \alpha_N \left( 1+ \epsilon \right) |u_{\epsilon}|^{ \frac{N}{N-1} } \right) dx.
	\end{flalign*}
\end{lem}

\begin{proof}
	A straightforward calculation shows
	\begin{flalign*}
		& \int_{\mathbb{B}_{R r_{\epsilon}^{ \frac{1}{1+ \epsilon} }} } |x|^{N \epsilon} \Phi \left( \alpha_N \left( 1+ \epsilon \right) |u_{\epsilon}|^{ \frac{N}{N-1} } \right) dx\\
		& = \lambda_{\epsilon}^{\phi} c_{\epsilon}^{ - \frac{N}{N-1} } \int_{\mathbb{B}_R} |y|^{N \epsilon} \exp \left\lbrace \alpha_N \left( 1+ \epsilon \right) \left( |\tilde{u}_{\epsilon}|^{ \frac{N}{N-1} } - c_{\epsilon}^{ \frac{N}{N-1} } \right) \right\rbrace -\\
		& \lambda_{\epsilon}^{\phi} c_{\epsilon}^{ - \frac{N}{N-1} } \sum_{n=0}^{N-2} \int_{\mathbb{B}_R} |y|^{N \epsilon} \frac{\alpha_N^n \left( 1+ \epsilon \right)^n |\tilde{u}_{\epsilon}|^{ \frac{nN}{N-1} } }{n !} dy.
	\end{flalign*}
\end{proof}

\begin{lem}
	We have that
	\begin{flalign*}
		& c_{\epsilon}^{ \frac{1}{N-1} } u_{\epsilon} \rightharpoonup G, ~ W_0^{1,p} \left( \mathbb{B} \right), ~ 1< p <N;\\
		& c_{\epsilon}^{ \frac{1}{N-1} } u_{\epsilon} \rightarrow G, ~ L^s \left( \mathbb{B} \right), ~ 1< s < \frac{Np}{N-p};\\
		& c_{\epsilon}^{ \frac{1}{N-1} } u_{\epsilon} \rightarrow G, ~  C_{loc}^1 \left( \overline{\mathbb{B}} \setminus \{0\} \right);
	\end{flalign*}
	where $G$ is a weak solution of
	\begin{flalign*}
		-\operatorname{div} \left( |\nabla G|^{N-2} \nabla G \right) = \delta_0.
	\end{flalign*}
	In addition, we also have
	\begin{flalign*}
		G = - \frac{N}{\alpha_N} \ln R + A_0 + O \left( r^N \ln R^N \right).
	\end{flalign*}
\end{lem}

\begin{proof}
	We divide the proof into two steps. Then, the first shows that, on any bounded domain $\Omega \subset\subset \mathbb{R}^N \setminus \{0\}$, we have that $c_{\epsilon}^{ \frac{1}{N-1} } u_{\epsilon} \rightarrow G$ in $C^1 \left( \Omega \right)$, where $G \in C^{1, \alpha}_{loc} \left( \mathbb{R}^N \setminus \{0\} \right)$ satisfies the following equation,
	\begin{flalign*}
		- \operatorname{div} \left( |\nabla G|^{N-2} \nabla G \right) = \delta_0.
	\end{flalign*} 
    The second shows that $G \in C^{1, \alpha}_{loc} \left( \mathbb{R}^N \setminus \{0\} \right)$ and near $0$ can be expressed as follows,
    \begin{flalign*}
    	G = - \frac{N}{ \alpha_N } \ln R + A_0 + O \left( R^N \ln^N R \right).
    \end{flalign*}
    Moreover, we have that for any $\delta >0$,
    \begin{flalign*}
    	\underset{\epsilon \rightarrow 0}{\lim} \int_{ \mathbb{B} \setminus \mathbb{B}_{\delta} } |\nabla U_{\epsilon}|^N dx & = \int_{\mathbb{B} \setminus \mathbb{B}_{\delta}} |\nabla G|^N dx\\
    	& = G\left( \delta \right).
    \end{flalign*}
    \textbf{Step 1}. Denote by $U_{\epsilon} := c_{\epsilon}^{ \frac{1}{N-1} } u_{\epsilon}$, which satisfy the equations,
     \begin{flalign*}
     	-\operatorname{div} \left( |\nabla U_{\epsilon}|^{N-2} \nabla U_{\epsilon} \right) &= \frac{ c_{\epsilon} }{ \lambda_{\epsilon}^{\phi} } |x|^{N \epsilon} |u_{\epsilon}|^{ \frac{1}{N-1} } \Phi^{\prime} \left( \alpha_N \left( 1+ \epsilon \right) |u_{\epsilon}|^{ \frac{N}{N-1} } \right).
     \end{flalign*}
    We show that,  for any $R\leq 1$ and $p>1$,
    \begin{flalign*}
    	\int_{\mathbb{B}_R} |\nabla U_{\epsilon}|^p dx < + \infty.
    \end{flalign*}
    Then, we consider $\Omega_t := \left\lbrace x\in \mathbb{B}_R: ~ 0 \leq U_{\epsilon } \leq t \right\rbrace $ and $U_{\epsilon}^t := \min \left\lbrace U_{\epsilon}, ~ t \right\rbrace$,
    \begin{flalign*}
    	\int_{\Omega_t} |\nabla U_{\epsilon}^t|^N dx & = - \int_{\Omega_t} U_{\epsilon}^t \operatorname{div} \left( |\nabla U_{\epsilon}|^{N-2} \nabla U_{\epsilon} \right) dt\\
    	& \leq - \int_{\mathbb{B}} U_{\epsilon}^t \operatorname{div} \left( |\nabla U_{\epsilon}|^{N-2} \nabla U_{\epsilon} \right) dt\\
    	& = \int_{\mathbb{B}} U_{\epsilon}^t \frac{ c_{\epsilon} }{ \lambda_{\epsilon}^{\phi} } |x|^{N \epsilon} |u_{\epsilon}|^{ \frac{1}{N-1} } \Phi^{\prime} \left( \alpha_N \left( 1+ \epsilon \right) |u_{\epsilon}|^{ \frac{N}{N-1} } \right) dx\\
    	& \leq C(t) \cdot t.
    \end{flalign*}
    Let $\eta$ be a radially symmetric cut-off function which is $1$ on $\mathbb{B}_R$ and $0$ on $\mathbb{B}_{2R}^c$ for any $R \leq \frac{1}{2}$. Then,
    \begin{flalign*}
    	\int_{\mathbb{B}_{2R}} |\nabla \eta U_{\epsilon}^t|^N dx \leq C_1 (R) + C_2 (R) t. 
    \end{flalign*}
    Let $\rho > 0 $ be such that $ U_{\epsilon} \left( \rho \right) = t $. From \cite{8}, we have that the infimum of 
    \begin{flalign*}
    	\inf \left\lbrace \int_{\mathbb{B}_{2R} } |\nabla u|^N dx: ~ u \in W_0^{1,N} \left( \mathbb{B}_{2R} \right), ~ u\rvert_{ \mathbb{B}_{\rho} } = t \right\rbrace \leq C_1 (R) + C_2 (R) t,
    \end{flalign*}
    can be attained by $- t \ln \frac{|x|}{2 R}/ \ln \frac{2 R}{\rho}$. By a direct computation, we have
    \begin{flalign*}
    	\frac{ \omega_{N-1} t^{N-1} }{ \left( \ln \frac{2R}{\rho} \right)^{N-1} } \leq 2 C_2(R),
    \end{flalign*}
    and hence, for any $t > \frac{C_1 (R)}{ C_2 (R) }$,
    \begin{flalign*}
    	\left| \left\lbrace x \in \mathbb{R}_{2 R}, ~ U_{\epsilon} \geq t  \right\rbrace \right| = |\mathbb{B}_{\rho}| \leq C_3 (R) e^{ - A(R) t} ,
    \end{flalign*}
    where $A(R) >0$ is a constant depending on $R$. For any $0< \alpha < A(R)$,
    \begin{flalign*}
    	\int_{ \mathbb{B}_{R} } e^{ \alpha U_{\epsilon} } dx & \leq \sum_{m=0}^{+ \infty} \left| m \leq U_{\epsilon} < m+1 \right| e^{ \alpha \left( m+1 \right) }\\
    	& \leq \sum_{m=0}^{+ \infty} e^{ - \left( A(R) - \alpha \right) m} e^{\alpha} \leq C .
    \end{flalign*}
    Consider the test function $\ln \frac{1 + 2 \left( U_{\epsilon} - U_{\epsilon} \left( R \right) \right)^+}{ 1 + \left( U_{\epsilon} - U_{\epsilon} \left( R \right) \right)^+ }$,
    \begin{flalign*}
    	& \int_{ \mathbb{B}_R } \frac{ |\nabla U_{\epsilon}|^N }{ \left( 1+ U_{\epsilon} - U_{\epsilon} (R)\right) \left( 1 + 2 U_{\epsilon} - 2 U_{\epsilon} (R)\right) } dx \\
    	& \leq \ln 2 \int_{\mathbb{B}_R} \frac{ c_{\epsilon} u_{\epsilon}^{ \frac{1}{N-1} } }{ \lambda_{\epsilon} } \Phi^{\prime} \left( \alpha_N \left( 1+ \epsilon \right) |u_{\epsilon}|^{ \frac{N}{N-1} } \right) dx\\
    	& < + \infty.
    \end{flalign*}
    By Young's inequality, for $p <N$, we have
    \begin{flalign*}
    	& \int_{ \mathbb{B}_R} |\nabla U_{\epsilon}|^p dx\\
    	& \leq \int_{\mathbb{B}_R} \left\lbrace \frac{ |\nabla U_{\epsilon}|^N }{ \left( 1+ U_{\epsilon} - U_{\epsilon} (R)\right) \left( 1 + 2 U_{\epsilon} - 2 U_{\epsilon} (R)\right) } + \left( \left( 1+ U_{\epsilon} \right) \left( 1 + 2 U_{\epsilon} \right) \right)^{ \frac{N}{N-p} } \right\rbrace dx\\
    	& \leq \int_{\mathbb{B}_R} \left\lbrace \frac{ |\nabla U_{\epsilon}|^N }{ \left( 1+ U_{\epsilon} - U_{\epsilon} (R)\right) \left( 1 + 2 U_{\epsilon} - 2 U_{\epsilon} (R)\right) } + C e^{ \alpha U_{\epsilon} } \right\rbrace dx.
    \end{flalign*}
    Then we assume that
    \begin{flalign*}
    	U_{\epsilon} \rightharpoonup G, ~ W_0^{1,p} \left( \mathbb{B}_R \right),
    \end{flalign*}
    for any $p<N$ and $0 < R<\frac{1}{2}$. From then on, the function is the Green function. Hence, we can get that $U_{\epsilon}$ and $e^{ \alpha_N \left( 1+ \epsilon \right) |u_{\epsilon}|^{ \frac{N}{n-1} } }$ are bounded in $L^p$ for any $p>0$. From the results in \cite{5}, we have that $\left\| U_{\epsilon } \right\|_{ C^{1, \alpha} } \leq C$. Then $U_{\epsilon} $ converges to $G$ in $C^1 \left( \mathbb{B}_R \right)$.
    
    \textbf{Step 2}. We know that the Green function  is given by
    \begin{flalign*}
    	G = - \frac{N}{ \alpha_N } \ln |x| + A_0 + o(1).
    \end{flalign*}
    From the equation and Green Theorem, we have that
    \begin{flalign*}
    	\omega_{N-1} G^{\prime} (r)^{N-1} r^{N-1} = \int_{ \partial \mathbb{B}_{r} } |\nabla G|^{N-2} \frac{\partial G}{ \partial \mathbf{n} } dS = 1.
    \end{flalign*}
    On the other hand, we also have
    \begin{flalign*}
    	\int_{ \mathbb{B} \setminus \mathbb{B}_{\delta} } |u_{\epsilon}|^{ \frac{N}{n-1} } \Phi^{\prime} \left( \alpha_N \left( 1+ \epsilon \right) |u_{\epsilon}|^{ \frac{N}{N-1} } \right) dx \rightarrow 0,
    \end{flalign*}
    as $\epsilon \rightarrow 0$. Recall that $U_{\epsilon} \in W_0^{1,N} \left( \mathbb{B} \right)$. Then
    \begin{flalign*}
    	\int_{ \mathbb{B} \setminus \mathbb{B}_{\delta} } |\nabla U_{\epsilon}|^N dx = \frac{ c_{\epsilon} }{ \lambda_{\epsilon}^{\phi} } \int_{ \mathbb{B} \setminus \mathbb{B}_{\delta} } |x|^{N \epsilon} |u_{\epsilon}|^{ \frac{N}{n-1} } \Phi^{\prime} \left( \alpha_N \left( 1+ \epsilon \right) |u_{\epsilon}|^{ \frac{N}{n-1} } \right) dx.
    \end{flalign*}
    \begin{flalign*}
    	\int_{ \mathbb{B} \setminus \mathbb{B}_{\delta} } |\nabla U_{\epsilon}|^Ndx & = - \int_{ \mathbb{B} \setminus \mathbb{B}_{\delta} } U_{\epsilon} \cdot \operatorname{div} \left( |\nabla U_{\epsilon}|^{N-2} \nabla U_{\epsilon} \right) dx + \\
    	& \int_{ \partial  \left( \mathbb{B} \setminus \mathbb{B}_{\delta} \right) } U_{\epsilon} \cdot |\nabla U_{\epsilon}|^{N-2} \frac{\partial U_{\epsilon} }{ \partial \mathbf{n}} dS\\
    	& = \frac{ c_{\epsilon} }{ \lambda_{\epsilon}^{\phi} } \int_{ \mathbb{B} \setminus \mathbb{B}_{\delta} } |x|^{N \epsilon} |u_{\epsilon}|^{ \frac{N}{n-1} } \Phi^{\prime} \left( \alpha_N \left( 1+ \epsilon \right) |u_{\epsilon}|^{ \frac{N}{n-1} } \right) dx - \\
    	& \int_{\partial \mathbb{B}_{\delta} } U_{\epsilon} \cdot |\nabla U_{\epsilon}|^{N-2} \frac{\partial U_{\epsilon} }{ \partial \mathbf{n}} dS.
    \end{flalign*}
    Then,
    \begin{flalign*}
    	\underset{\epsilon \rightarrow 0}{\lim} \int_{ \mathbb{B} \setminus \mathbb{B}_{\delta} } |\nabla U_{\epsilon}|^N dx & = - \underset{\epsilon \rightarrow 0}{\lim} \int_{ \partial \mathbb{B}_{\delta} } U_{\epsilon} \cdot |\nabla U_{\epsilon}|^{N-2} \frac{\partial U_{\epsilon} }{ \partial \mathbf{n} } dS\\
    	& = - G \left( \delta \right) \int_{ \partial \mathbb{B}_{\delta} } |\nabla G|^{N-2} \frac{\partial G}{\partial \mathbf{n} } dS\\
    	& = G \left( \delta \right).
    \end{flalign*}
\end{proof}

\begin{lem}
	\begin{flalign*}
		& \underset{u \in W_0^{1,N} \left( \mathbb{B} \right) \cap \mathcal{S} \setminus \{0\}, ||\nabla u||_{L^N} \leq 1}{\sup} \int_{\mathbb{B}} |x|^{N \epsilon}\Phi \left( \alpha_N \left( 1+ \epsilon \right) |u|^{ \frac{N}{N-1} } \right)dx \\
		& \leq \min \left\lbrace \frac{\alpha_N^{N-1}}{(N-1)!}, \frac{\omega_{N-1}}{N} e^{\alpha_N A_0 + \sum_{j=0}^{N-1} \frac{1}{j} } \right\rbrace. 
	\end{flalign*}
\end{lem}

\begin{proof}
	If $\underset{\epsilon > 0}{\sup} c_{\epsilon} < + \infty$, then we have
	\begin{flalign*}
		& \underset{\epsilon \rightarrow 0}{\lim} \int_{ \mathbb{B} } |x|^{N\epsilon} \left\lbrace  \Phi \left( \alpha_N \left( 1+ \epsilon \right) |u_{\epsilon}|^{ \frac{N}{N-1} } \right) -\frac{ \alpha_N^{N-1} \left( 1+ \epsilon \right)^{N-1} |u_{\epsilon}|^N }{ \left( N-1 \right)! } \right\rbrace  dx \\
		&= \int_{ \mathbb{B} } \left\lbrace \Phi \left( \alpha_N |u|^{ \frac{N}{N-1} } \right) - \frac{\alpha_N^{N-1} |u|^{ \frac{N}{N-1} } }{\left( N-1 \right)!} \right\rbrace dx,
	\end{flalign*}
    where $u$ is the weak limit of $u_{\epsilon}$. Then we have either 
    \begin{flalign*}
    	\underset{\epsilon \rightarrow 0}{\lim} \int_{ \mathbb{B} } |x|^{N \epsilon} \Phi \left( \alpha_N \left( 1+ \epsilon \right) |u_{\epsilon}|^{ \frac{N}{N-1} } \right) dx = \int_{ \mathbb{B} } \Phi \left( \alpha_N  |u|^{ \frac{N}{N-1} } \right) dx,
    \end{flalign*}
    or
	\begin{flalign*}
		& \underset{u \in W_0^{1,N} \left( \mathbb{B} \right) \cap \mathcal{S} \setminus \{0\}, ||\nabla u||_{L^N} \leq 1}{\sup} \int_{\mathbb{B}} |x|^{N \epsilon}\Phi \left( \alpha_N \left( 1+ \epsilon \right) |u|^{ \frac{N}{N-1} } \right)dx \\
		&\leq \frac{\alpha_N^{N-1}}{(N-1)!}.
	\end{flalign*}
    If $c_{\epsilon}$ is not bounded, then there exist some Green functions to express the behaviour of blow up. Then we only need to show that, if the function sequence $\{ u_{\epsilon} \}_{\epsilon} \subset W_0^{1,N} \left( \mathbb{B} \right) \cap \mathcal{S} \setminus \{0\}$ with $||\nabla u_{\epsilon}||_{L^N} \leq 1$ satisfies $u_{\epsilon} \rightharpoonup 0$ in $W_0^{1,N} \left( \mathbb{B} \right)$, then
    \begin{flalign*}
    	\underset{\epsilon \rightarrow 0}{\lim \sup} \int_{\mathbb{B} } |x|^{N \epsilon} \left( e^{ \alpha_N \left( 1+ \epsilon \right) |u_{\epsilon}|^{ \frac{N}{N-1} } } - 1 \right) dx \leq | \mathbb{B}| e^{ \sum_{j=1}^{N-1} \frac{1}{j} }.
    \end{flalign*}
    We know that
    \begin{flalign*}
    	\underset{\epsilon \rightarrow 0}{\lim} \int_{ \mathbb{B} \setminus \mathbb{B}_{\delta} }  \left|\nabla \left( c_{\epsilon}^{ \frac{1}{N-1}} u_{\epsilon}\right) \right|^N  dx = G \left( \delta \right).
    \end{flalign*}
    On the other hand, we set $v_{\epsilon} = \frac{ \left( u_{\epsilon} \left(x \right) - u_{\epsilon} \left( \delta \right) \right)^+ }{ \left\| \nabla u_{\epsilon} \right\|_{ L^N \left( \mathbb{B}_{\delta} \right) } } \in W_0^{1,N} \left( \mathbb{B}_{\delta} \right)$ for any $\epsilon >0$. Then we have
    \begin{flalign*}
    	\underset{\epsilon \rightarrow 0}{\lim \sup} \int_{ \mathbb{B}_{\delta} } |x|^{N \epsilon} e^{\alpha_N \left(1 + \epsilon \right) |v_{\epsilon}|^{ \frac{N}{N-1} } } dx \leq |\mathbb{B}_{\delta}|\left( 1+ e^{ \sum_{j=1 }^{N-1} \frac{1}{j} }\right).
    \end{flalign*}
    Then we have
    \begin{flalign*}
    	\int_{ \mathbb{B}_{\delta} } |\nabla u_{\epsilon}|^N dx & = 1- \int_{ \mathbb{B} \setminus \mathbb{B}_{\delta} } |\nabla u_{\epsilon} |^N dx\\
    	& = 1 - \frac{G \left( \delta \right)}{ c_{\epsilon}^{ \frac{N}{N-1} } } + o_{\epsilon} \left(\delta \right),
    \end{flalign*}
    where $\underset{\epsilon \rightarrow 0}{\lim} \underset{\delta \rightarrow 0}{\lim} o_{\epsilon} \left( \delta \right) =0 $.
    From the estimate and the definition, we have
    \begin{flalign*}
    	\left( v_{\epsilon} \right)^{ \frac{N}{N-1} } & \leq \frac{ |u_{\epsilon}|^{ \frac{N}{N-1} } }{ \left( 1 - \frac{ G \left( \delta \right) }{ c_{\epsilon}^{ \frac{N}{N-1} } } + o_{\epsilon} \left( \delta \right) \right)^{ \frac{1}{N-1} }  }\\
    	& \leq |u_{\epsilon}|^{ \frac{N}{N-1} } - \frac{ \ln \delta^N }{ \left( N-1 \right) \alpha_N },
    \end{flalign*}
    in $\mathbb{B}_{\delta}$. Then we have
    \begin{flalign*}
    	e^{ \alpha_N \left( 1+ \epsilon \right) |v_{\epsilon}|^{ \frac{N}{N-1} } } & \leq e^{ \alpha_N \left( 1+ \epsilon \right) |u_{\epsilon}|^{ \frac{N}{N-1} } - \frac{\ln \delta^N}{ \left( N -1 \right) \alpha_N } }\\
    	& = O \left( \delta^{-N} \right) e^{ \alpha_N \left( 1+ \epsilon \right) |u_{\epsilon}|^{ \frac{N}{N-1} } },
    \end{flalign*}
    moreover, 
    \begin{flalign*}
    	&\underset{R \rightarrow + \infty}{\lim} \underset{\epsilon \rightarrow 0}{\lim} \int_{ \mathbb{B}_{\delta} \setminus \mathbb{B}_{R r_{\epsilon}^{ \frac{1}{1+ \epsilon} } } } |x|^{N \epsilon} e^{ \alpha_N \left( 1+ \epsilon \right) |v_{\epsilon}|^{ \frac{N}{N-1} } } dx \\
    	& \leq O \left( \delta^N \right) \underset{R \rightarrow + \infty}{\lim} \underset{\epsilon \rightarrow 0}{\lim} \int_{ \mathbb{B}_{\delta}  \setminus \mathbb{B}_{R r_{\epsilon}^{ \frac{1}{1+ \epsilon} } } } |x|^{N \epsilon} e^{ \alpha_N \left( 1+ \epsilon \right) |u_{\epsilon}|^{ \frac{N}{N-1} } } dx\\
    	& = |\mathbb{B}_{\delta}| O \left( \delta^{-N} \right).
    \end{flalign*}
    Since $v_{\epsilon} \rightarrow 0$ in $\mathbb{B}_{\delta}  \setminus \mathbb{B}_{R r_{\epsilon}^{ \frac{1}{1+ \epsilon} } }$, we have
    \begin{flalign*}
    	\underset{\epsilon \rightarrow 0}{\lim} \int_{ \mathbb{B}_{\delta}  \setminus \mathbb{B}_{R r_{\epsilon}^{ \frac{1}{1+ \epsilon} } } } |x|^{N \epsilon} \left( e^{ \alpha_N \left( 1+ \epsilon \right) |v_{\epsilon}|^{ \frac{N}{n-1} } } - 1\right) dx =0.
    \end{flalign*}
    Moreover, 
    \begin{flalign*}
    	\underset{R \rightarrow + \infty}{\lim } \underset{\epsilon \rightarrow 0}{\lim} \int_{ \mathbb{B}_{\delta}  \setminus \mathbb{B}_{R r_{\epsilon}^{ \frac{1}{1+ \epsilon} } } } |x|^{N \epsilon} \left( e^{ \alpha_N \left( 1+ \epsilon \right) |v_{\epsilon}|^{ \frac{N}{n-1} } } - 1\right) dx \leq O \left( \delta^{-N} \right) |\mathbb{B}_{\rho} |, 
    \end{flalign*}
    for any $\rho < \delta$. Letting $\rho \rightarrow 0$, we get
    \begin{flalign*}
    	\underset{R \rightarrow + \infty}{\lim } \underset{\epsilon \rightarrow 0}{\lim} \int_{ \mathbb{B}_{\delta}  \setminus \mathbb{B}_{R r_{\epsilon}^{ \frac{1}{1+ \epsilon} } } } |x|^{N \epsilon} \left( e^{ \alpha_N \left( 1+ \epsilon \right) |v_{\epsilon}|^{ \frac{N}{n-1} } } - 1\right) dx =0.
    \end{flalign*}
    Then,
    \begin{flalign*}
    	\underset{R \rightarrow + \infty}{\lim } \underset{\epsilon \rightarrow 0}{\lim} \int_{ \mathbb{B}_{R r_{\epsilon}^{ \frac{1}{1+ \epsilon} } } } |x|^{N \epsilon} \left( e^{ \alpha_N \left( 1+ \epsilon \right) |v_{\epsilon}|^{ \frac{N}{n-1} } } - 1\right) dx \leq |\mathbb{B}_{\delta}| e^{ \sum_{j=1}^{N-1} \frac{1}{j} }.
    \end{flalign*}
    For any fixed $R>0$ and any $x \in \mathbb{B}_{R r_{\epsilon}^{ \frac{1}{1 + \epsilon} } }$, we have 
    \begin{flalign*}
    	\alpha_N \left( 1+ \epsilon \right) |u_{\epsilon}|^{ \frac{N}{N-1} } & \leq \alpha_N \left( 1+ \epsilon \right) \left( \frac{ u_{\epsilon} }{ ||\nabla u_{\epsilon}||_{ L^N \left( \mathbb{B}_{\delta} \right) } } \right)^{ \frac{N}{N-1} } \left( \int_{ \mathbb{B}_{\delta} } |\nabla u_{\epsilon}|^N dx \right)^{ \frac{1}{N-1} }\\
    	& \leq \alpha_N \left( 1+ \epsilon \right) \left( v_{\epsilon} + \frac{ u_{\epsilon} \left( \delta \right) }{ ||\nabla u_{\epsilon}||_{ L^N \left( \mathbb{B}_{\delta} \right) } } \right)^{ \frac{N}{N-1} } \left( \int_{ \mathbb{B}_{\delta} } |\nabla u_{\epsilon}|^N dx \right)^{ \frac{1}{N-1} }.
    \end{flalign*}
    Also, we have 
    \begin{flalign*}
    	& u_{\epsilon} \left( \delta \right) = O \left( \frac{1}{ c_{\epsilon}^{ \frac{1}{N-1} } } \right),\\
    	& ||\nabla u_{\epsilon}||_{L^N \left( \mathbb{B}_{\delta} \right) } = 1+ O \left( \frac{1}{ c_{\epsilon}^{ \frac{N}{N-1} }  } \right).
    \end{flalign*}
    Then, 
    \begin{flalign*}
    	\alpha_N \left( 1+ \epsilon \right) |u_{\epsilon}|^{ \frac{N}{N-1} } & \leq \alpha_N \left( 1+ \epsilon \right) \left( v_{\epsilon} + u_{\epsilon} \left( \delta \right) + O \left(  c_{\epsilon}^{ - \frac{N+1}{N-1} } \right) \right)^{ \frac{N}{N-1} } \times \\
    	& \left( \int_{ \mathbb{B}_{\delta} } |\nabla u_{\epsilon}|^N dx \right)^{ \frac{1}{N-1} }\\
    	& \leq \alpha_N \left( 1+ \epsilon \right) |v_{\epsilon}|^{ \frac{N}{N-1} } \left( 1+ \frac{ u_{\epsilon} \left( \delta \right) }{ v_{\epsilon} } + O \left( c_{\epsilon}^{ - \frac{2N}{N-1} } \right) \right)^{ \frac{N}{N-1} } \times\\
    	& \left( 1 - \frac{G \left( \delta \right) + o_{\epsilon} \left( \delta \right) }{ c_{\epsilon}^{ \frac{N}{N-1} } } \right)^{ \frac{1}{N-1} }\\
    	& = \alpha_N \left( 1+ \epsilon \right) |v_{\epsilon}|^{ \frac{N}{N-1} } \times\\
    	& \left\lbrace 1 + \frac{N}{N-1} \frac{ u_{\epsilon} \left( \delta \right) }{ v_{\epsilon} } - \frac{1}{N-1} \frac{ G \left( \delta \right) + o_{\epsilon} \left( \delta \right) }{ c_{\epsilon}^{ \frac{N}{N-1} } } + O \left( c_{\epsilon}^{ -\frac{2N}{N-1} } \right) \right\rbrace .
    \end{flalign*}
    On the other hand, we have
    \begin{flalign*}
    	& \frac{ v_{\epsilon} \left( r_{\epsilon}^{ \frac{1}{N-1} } x \right) }{ c_{\epsilon} } \rightarrow 1, \\
    	& v_{\epsilon}^{ \frac{1}{N-1} } \left( r_{\epsilon}^{ \frac{1}{ 1+ \epsilon} } x \right) u_{\epsilon} \left( \delta \right) \rightarrow G \left( \delta \right).
    \end{flalign*}
    Therefore, we have
    \begin{flalign*}
    	&\underset{R \rightarrow + \infty}{\lim} \underset{\epsilon \rightarrow 0}{\lim} \int_{ \mathbb{B}_{ R r_{\epsilon}^{ \frac{1}{1+ \epsilon} } } } |x|^{N \epsilon} \left( e^{ \alpha_N \left( 1+ \epsilon \right) |u_{\epsilon}|^{ \frac{N}{N-1} } } - 1\right) dx \\
    	&= \underset{R \rightarrow + \infty}{\lim} \underset{\epsilon \rightarrow 0}{\lim} e^{ \alpha_N G\left( \delta \right) } \int_{ \mathbb{B}_{ R r_{\epsilon}^{ \frac{1}{1+ \epsilon} } } } |x|^{N \epsilon} \left( e^{ \alpha_N \left( 1+ \epsilon \right) |v_{\epsilon}|^{ \frac{N}{N-1} } } -1 \right) dx\\
    	& \leq e^{ \alpha_N G \left( \delta \right) } \delta^N \frac{\omega_{N-1}}{N} e^{ \sum_{j=1}^{N-1} \frac{1}{j} }.
    \end{flalign*}
    Letting $\delta \rightarrow 0$, we complete the proof.
\end{proof}

\begin{lem}
	There exists a sequence $\{ u_{\epsilon}\}_{\epsilon} \subset W_0^{1,N} \left( \mathbb{B} \right) \cap \mathcal{S} \setminus \{0\}$ such that
	\begin{flalign*}
		& \int_{\mathbb{B}} |x|^{N \epsilon}\Phi \left( \alpha_N \left( 1+ \epsilon \right) |u_k|^{ \frac{N}{N-1} } \right)dx\\
		& > \min \left\lbrace \frac{\alpha_N^{N-1}}{(N-1)!}, \frac{\omega_{N-1}}{N} e^{\alpha_N A_0 + \sum_{j=1}^{N-1} \frac{1}{j} } \right\rbrace. 
	\end{flalign*}
\end{lem}

\begin{proof}
	We only need to find two sequences $\{ u_{\epsilon}^{\prime} \}_{\epsilon}$ and $\{ u_{\epsilon}^{ \prime \prime} \}_{\epsilon}$ satisfying
	\begin{flalign*}
		& \int_{\mathbb{B}} |x|^{N \epsilon} \Phi \left( \alpha_N \left( 1+ \epsilon \right) |u_{\epsilon}^{\prime}|^{ \frac{N}{N-1} } \right) dx > \frac{\alpha_N^{N-1}}{\left( N-1 \right)!},\\
		& \int_{\mathbb{B}} |x|^{N \epsilon} \Phi \left( \alpha_N \left( 1+ \epsilon \right) |u_{\epsilon}^{\prime \prime}|^{ \frac{N}{N-1} } \right) dx > \frac{\omega_{N-1}}{N} e^{ \alpha_N A_0 + \sum_{j=1}^{N-1} \frac{1}{j} }.
	\end{flalign*}
    \textbf{Part 1}.
    We consider the functions as follows,
    \begin{flalign*}
    	u_{\epsilon}^{ \prime } (x) =
    	\begin{cases}
    		& C - \frac{ \left( N-1 \right) \ln \left( 1 + c_N \left| \frac{x}{\epsilon} \right|^{ \frac{N}{N-1} } \right) + A_{\epsilon} }{  \alpha_N C^{ \frac{1}{N-1} } }, ~ |x| \leq R \epsilon,\\
    		& \frac{ G \left( |x| \right) }{ C^{ \frac{1}{N-1} } }, ~ |x| \geq R \epsilon,
    	\end{cases}
    \end{flalign*}
    where,
    \begin{flalign*}
    	& C \rightarrow + \infty,\\
    	& R \rightarrow + \infty,\\
    	& R \epsilon \rightarrow 0,
    \end{flalign*}
    as $ \epsilon \rightarrow 0$. Also,
    \begin{flalign} \label{eq: 4.1}
    	C - \frac{ \left( N-1 \right) \ln \left( 1 + c_N R^{ \frac{N}{N-1} } \right) + A_{\epsilon} }{  \alpha_N C^{ \frac{1}{N-1} } } = \frac{ G \left(R \epsilon \right) }{ C^{ \frac{1}{N-1} } }.
    \end{flalign}
    From \ref{eq: 4.1} and $\int_{\mathbb{B}} |\nabla u_{\epsilon}^{\prime}|^N dx =1$, we have
    \begin{flalign*}
    	\int_{\mathbb{B}_{R \epsilon} } |\nabla u_{\epsilon}^{\prime}|^N dx + \int_{ \mathbb{B} \setminus \mathbb{B}_{R \epsilon} } |\nabla u_{\epsilon}^{\prime}|^N dx = 1.
    \end{flalign*}
    \begin{flalign*}
    	\int_{ \mathbb{B} \setminus \mathbb{B}_{R \epsilon} } |\nabla u_{\epsilon}^{\prime}|^N dx & = \frac{1}{ C^{ \frac{N}{N-1} } } \int_{ \mathbb{B} \setminus \mathbb{B}_{R \epsilon} } |\nabla G|^N dx\\
    	& =  \frac{1}{ C^{ \frac{N}{N-1} } } G\left( R \epsilon \right).
    \end{flalign*}
    \begin{flalign*}
    	\int_{ \mathbb{B}_{R \epsilon} } |\nabla u_{\epsilon}^{\prime} |^N dx & = \left( \frac{N-1}{ \alpha_N C^{ \frac{1}{N-1} } } \right)^N \int_{ \mathbb{B}_{R \epsilon} } \left| \nabla \ln \left( 1 + c_N \left| \frac{x}{ \epsilon } \right|^{ \frac{N}{N-1} } \right) \right|^N dx\\
    	& = \left( \frac{N-1}{ \alpha_N C^{ \frac{1}{N-1} } } \right)^N \omega_{N-1} \int_0^{ R \epsilon } r^{N-1} \left| \nabla \ln \left( 1 + c_N \left| \frac{r}{ \epsilon } \right|^{ \frac{N}{N-1} } \right) \right|^N dr.
    \end{flalign*}
    By direct calculation, we obtain
    \begin{flalign*}
    	\int_{ \mathbb{B}_{R \epsilon} } |\nabla u_{\epsilon}^{\prime}|^N dx &= \frac{N-1}{\alpha_N \cdot C^{ \frac{N}{N-1} } } \int_0^{ c_N \left( R \right)^{ \frac{N}{N-1} } } \frac{u^{N-1}}{ \left( 1+ u \right)^N } du.
    \end{flalign*}
    From the known result
    \begin{flalign*}
    	- \sum_{k=0}^{N-2} \frac{ C_{N-1}^k \left( -1 \right)^{ N-1-k} }{N-k-1} = \sum_{k=1}^{N-1} \frac{1}{k},
    \end{flalign*}
    it follows that
    \begin{flalign*}
    	\int_{ \mathbb{B}_{R \epsilon} } |\nabla u_{\epsilon}^{\prime} |^N dx & = \frac{N-1}{\alpha_N C^{ \frac{N}{N-1} } } \left( \ln \left( 1+ c_N R^{ \frac{N}{N-1} } \right) - \sum_{j=1}^{N-1} \frac{1}{j}  \right) + O \left( \frac{1}{R^{ \frac{N}{N-1} } C^{ \frac{N}{N-1} } } \right).
    \end{flalign*}
    Then we have,
    \begin{flalign*}
    	1 & = \frac{1}{C^{ \frac{N}{N-1} } } \left( - \frac{N}{\alpha_N} \ln |R \epsilon| + A_0 + O \left( |R \epsilon|^N \ln^N |R \epsilon| \right) \right) +\\
    	& \frac{N-1}{\alpha_N C^{ \frac{N}{N-1} } } \left( \ln \left( 1+ c_N R^{ \frac{N}{N-1} } \right) - \sum_{j=1}^{N-1} \frac{1}{j}  \right) + O \left( \frac{1}{R^{ \frac{N}{N-1} } C^{ \frac{N}{N-1} } } \right).
    \end{flalign*}
    Also,  
    \begin{flalign*}
    	& C - \frac{\left( N-1 \right) \ln \left( 1+ c_N R^{ \frac{N}{N-1} } \right) + A_{\epsilon} }{ \alpha_N C^{ \frac{1}{N-1} } } \\
    	& = \frac{1}{ C^{ \frac{1}{N-1} } } \left( - \frac{N}{\alpha_N} \ln |R \epsilon| + A_0 + O\left( |R \epsilon|^N \ln^N |R \epsilon| \right)  \right).
    \end{flalign*}
    We have that,
    \begin{flalign*}
    	A_{\epsilon} & = - \left( N-1 \right) \sum_{j=1}^{N-1} \frac{1}{j} + O \left( |R \epsilon |^N \ln^N |R \epsilon | \right) + O \left( R^{ - \frac{N}{N-1} } \right).
    \end{flalign*}
    Next, we prove the following result,
    \begin{flalign*}
    	\int_{ \mathbb{B} } |x|^{N\epsilon} \Phi \left( \alpha_N \left( 1 + \epsilon \right) |u_{\epsilon}^{ \prime }|^{ \frac{N}{N-1} } \right) dx > \frac{\omega_{N-1} }{N} e^{ \alpha_N A_0 + \sum_{j=1}^{N-1} \frac{1}{j}}.
    \end{flalign*}
    From definition and $|1 - t|^{ \frac{N}{N-1} } \geq 1 - \frac{N}{N-1} t$ when $|t| < 1$, we have 
    \begin{flalign*}
    	& \int_{ \mathbb{B}_{R \epsilon} } |x|^{N\epsilon} e^{ \alpha_N \left( 1+ \epsilon \right) |u_{\epsilon}^{\prime}|^{ \frac{N}{N-1} } } dx \\
    	& \geq  \int_{ \mathbb{B}_{R \epsilon} } |x|^{N\epsilon} \exp \left\lbrace  \alpha_N \left( 1+ \epsilon \right) C^{ \frac{N}{N-1} } \left( 1 - \frac{ N \ln \left( 1 + c_N \left| \frac{x}{\epsilon} \right|^{ \frac{N}{N-1} } \right) +A_{\epsilon} }{\alpha_N C^{ \frac{N}{N-1} } } \right) \right\rbrace  dx\\
    	& = e^{ \alpha_N \left( 1+ \epsilon \right) C^{ \frac{N}{N-1} } - \left(1 + \epsilon \right) A_{\epsilon}} \times \\
    	& \int_{\mathbb{B}_{R}} \epsilon^{N \epsilon + N} |y|^{N \epsilon} \left( 1+ c_N |y|^{ \frac{N}{N-1} } \right)^{ - N \left( 1+\epsilon \right)}  dy.
    \end{flalign*}
    Moreover, we have 
    \begin{flalign*}
    	& \int_{\mathbb{B}_{R}} \epsilon^{N \epsilon + N} |y|^{N \epsilon} \left( 1+ c_N |y|^{ \frac{N}{N-1} } \right)^{ - N \left( 1+\epsilon \right)}  dy\\
    	&  = \epsilon^{N \epsilon + N} \omega_{N-1} \int_0^R \frac{r^{N \epsilon}}{ \left( 1+ c_N r^{ \frac{N}{N-1} } \right)^{N\left( 1+ \epsilon \right)} } r^{N-1} dr\\
    	& \geq \epsilon^{N \epsilon + N} \omega_{N-1}  \frac{1}{c_N^N} \frac{N-1}{N} \int_0^R \frac{ \left( c_N r^{ \frac{N}{N-1} }   \right)^{N-1}   }{ \left( 1+ c_N r^{ \frac{N}{N-1} } \right)^{ N \left( 1+ \epsilon \right) } } d \left( c_N r^{ \frac{N}{N-1} } \right).
    \end{flalign*}
    In addition, 
    \begin{flalign*}
    	& \exp \left\lbrace \alpha_N\left( 1+ \epsilon \right) C^{ \frac{N}{N-1} } - \left( 1+ \epsilon \right) A_{\epsilon} \right\rbrace \\
    	& = e^{1 + \epsilon} \times\\
    	& \exp \left\lbrace \left( N-1 \right) \sum_{j=1}^{N-1} \frac{1}{j} + \ln \frac{\omega_{N-1}}{N} - \ln \epsilon^{N}  + \alpha_N A_0 \right\rbrace \times\\
    	&  \exp \left\lbrace O \left( |R \epsilon|^N \ln^N |R \epsilon| \right) + O \left( R^{ -\frac{N}{N-1} } \right) \right\rbrace.
    \end{flalign*}
    Then 
    \begin{flalign*}
    	& \exp \left\lbrace \alpha_N \left( 1+ \epsilon \right) C^{ \frac{N}{N-1} } - \left( 1+ \epsilon \right) A_{\epsilon}  \right\rbrace \\
    	& \geq \frac{\omega_{N-1} }{N} e^{ \alpha_N A_0 + \sum_{j=1}^{N-1} \frac{1}{j} + O \left( R^{ - \frac{N}{N-1} }  \right) }.
    \end{flalign*}
    Then, 
    \begin{flalign*}
    	\int_{ \mathbb{B}_{R \epsilon} } |x|^{N \epsilon} e^{ \alpha_N\left( 1+ \epsilon \right) |u_{\epsilon}^{\prime}|^{ \frac{N}{N-1} } } dx & = \frac{\omega_{N-1}}{N} e^{ \alpha_N A_0 + \sum_{j=1}^{N-1} \frac{1}{j} } \\
    	& + O \left( |R \epsilon|^N \ln^N |R \epsilon| \right) + O \left( R^{ -\frac{N}{N-1} } \right).
    \end{flalign*}
    On the other hand, we have
    \begin{flalign*}
    	\int_{ \mathbb{B} \setminus \mathbb{B}_{R \epsilon} } |x|^{N\epsilon} \Phi \left( \alpha_N \left( 1+ \epsilon \right)|u_{\epsilon}^{\prime}|^{ \frac{N}{N-1} } \right) dx \geq \frac{\alpha_N^{N-1} \left(1 + \epsilon \right)^{N-1} }{\left( N-1 \right)!} \int_{ \mathbb{B} \setminus \mathbb{B}_{R \epsilon} } |x|^{N\epsilon} \left| \frac{G (x)}{ C^{ \frac{1}{N-1} } } \right|^{N} dx.
    \end{flalign*}
    Then we have 
    \begin{flalign*}
    	\int_{\mathbb{B}} |x|^{N\epsilon} \Phi \left( \alpha_N\left( 1+ \epsilon \right) |u_{\epsilon}^{\prime}|^{ \frac{N}{N-1} } \right) dx & \geq \frac{\omega_{N-1}}{N} e^{ \alpha_N A_0 + \sum_{j=1}^{N-1} \frac{1}{j} } \\
    	& + \frac{\alpha_N^{N-1} \left(1 + \epsilon \right)^{N-1} }{\left( N-1 \right)!} \int_{ \mathbb{B} \setminus \mathbb{B}_{R \epsilon} } |x|^{N\epsilon} \left| \frac{G (x)}{ C^{ \frac{1}{N-1} } } \right|^{N} dx\\
    	& + O \left( |R \epsilon|^N \ln^N |R \epsilon| \right) + O \left( R^{ -\frac{N}{N-1} } \right).
    \end{flalign*}
    We also have that
    \begin{flalign*}
    	\alpha_N C^{ \frac{N}{N-1} } &= - \left(N-1 \right) \sum_{j=1}^{N-1} \frac{1}{j} + \alpha_N A_0 -N \ln \epsilon + \ln \frac{\omega_{N-1} }{N} \\
    	& +O \left( |R \epsilon|^N  \ln^N |R \epsilon | \right) + O \left( R^{ -\frac{N}{N-1} } \right).
    \end{flalign*}
    Moreover, by standard calculus, there exists a constant $C>0$ such that
    \begin{flalign*}
    	\alpha_N C^{ \frac{N}{N-1} } = -N \ln \epsilon + O\left(1 \right).
    \end{flalign*}
    Letting $R = -\ln \epsilon$, it is obvious that
    \begin{flalign*}
    	\frac{\ln R}{C^{ \frac{N}{N-1} } } \rightarrow 0,
    \end{flalign*}
    as $\epsilon \rightarrow 0^+$. Then we have
    \begin{flalign*}
    	C^{ \frac{N}{N-1} } \ln |R \epsilon|^N |R \epsilon|^N + \frac{C^{ \frac{N}{N-1} } }{R^{ \frac{N}{N-1} }} \rightarrow 0,
    \end{flalign*}
    as $ \epsilon \rightarrow 0$. We consider the equation
    \begin{flalign*}
    	\int_{\mathbb{B}_{R \epsilon} } |x|^{N \epsilon} \sum_{j=0}^{N-2} \frac{ \alpha_N^j \left( 1+ \epsilon \right)^j |u_{\epsilon}^{\prime}|^{ \frac{jN}{N-1} }  }{j!} dx.
    \end{flalign*}
    While, we only consider the behaviour of the equation
    \begin{flalign*}
    	\int_{ \mathbb{B}_{R \epsilon} } |x|^{N \epsilon} |u_{\epsilon}^{\prime}|^{ \frac{\left( N -2 \right)N}{N-1} } dx.
    \end{flalign*}
    From direct calculation and the Lebesgue Dominated Convergence Theorem,
    \begin{flalign*}
    	\underset{\epsilon \rightarrow 0}{\lim \sup} \int_{ \mathbb{B}_{R \epsilon} } |x|^{N \epsilon} |u_{\epsilon}^{\prime}|^{ \frac{\left( N -2 \right)N}{N-1} } dx = 0.
    \end{flalign*}
    Then
    \begin{flalign*}
    	\int_{\mathbb{B}} |x|^{N \epsilon} \Phi \left( \alpha_N \left( 1+ \epsilon \right) |u_{\epsilon}^{\prime}|^{ \frac{N}{N-1} } \right) dx > \frac{\omega_{N-1}}{N} e^{ \alpha_N A_0 + \sum_{j=1}^{N-1} \frac{1}{j}}.
    \end{flalign*}
    \textbf{Part 2}. We consider the functions as follows,
    \begin{flalign*}
    	u_{\epsilon}^{ \prime \prime } (x) = 
    	\begin{cases}
    		&c, ~ |x| < \frac{R}{1+R} \epsilon,\\
    		&\frac{-N \ln \left| \frac{\left( 1+ R \right) x  }{R}\right| }{ \alpha_N c^{ \frac{1}{N-1} } }, ~ \frac{R}{1+R} \epsilon \leq |x| \leq \frac{R}{1+R},\\
    		& 0, ~ |x| > \frac{R}{1+R},
    	\end{cases}
    \end{flalign*}
    where $\epsilon^N = e^{ -\alpha_N c^{ \frac{N}{N-1} } }$. We have 
    \begin{flalign*}
    	\int_{\mathbb{B}} |\nabla u_{\epsilon}^{\prime \prime} |^N dx =1.
    \end{flalign*}
    Then it suffices to show that
    \begin{flalign*}
    	\int_{\mathbb{B}} |x|^{N \epsilon} \Phi \left(\alpha_N \left( 1+ \epsilon \right) \left| \frac{u_{\epsilon}^{\prime \prime} }{ \left\| u_{\epsilon}^{\prime \prime} \right\|_{ W^{1,N} } } \right|^{ \frac{N}{N-1} }\right) dx > \frac{\alpha_N^{N-1}}{\left(N-1 \right)!},
    \end{flalign*}
    where $\left\|u_{\epsilon}^{\prime \prime} \right\|_{W^{1,N} }^N = \left\| \nabla u_{\epsilon}^{\prime \prime} \right\|_{L^N}^N + \left\| u_{\epsilon}^{\prime \prime} \right\|_{L^N}^N$. We have that
    \begin{flalign*}
    	& \int_{\mathbb{B}} |x|^{N \epsilon} \Phi \left(\alpha_N \left( 1+ \epsilon \right) \left| \frac{u_{\epsilon}^{\prime \prime}}{ \left\| \nabla u_{\epsilon}^{\prime \prime} \right\|_{W^{1,N} } } \right|^{ \frac{N}{N-1} }\right) dx\\
    	& \geq \frac{\alpha_N^{N-1} \left( 1+ \epsilon \right)^{N-1}}{\left( N- 1\right)!} \int_{\mathbb{B}} |x|^{N \epsilon} \left| \frac{u_{\epsilon}^{\prime \prime}}{\left\| \nabla u_{\epsilon}^{\prime \prime} \right\|_{W^{1,N} }} \right|^N dx \\
    	&+\frac{\alpha_N^{N} \left( 1+ \epsilon \right)^{N}}{ N!} \int_{\mathbb{B}} |x|^{N \epsilon} \left| \frac{u_{\epsilon}^{\prime \prime}}{ \left\| \nabla u_{\epsilon}^{\prime \prime} \right\|_{W^{1,N} } } \right|^{ \frac{N^2}{N-1} } dx.
    \end{flalign*}
    On the other hand, we have
    \begin{flalign*}
    	& \int_{ \mathbb{B} } |x|^{N \epsilon} |u_{\epsilon}^{\prime \prime} |^N dx + \int_{ \mathbb{B} } |x|^{N \epsilon} |u_{\epsilon}^{ \prime \prime }|^{ \frac{N^2}{N-1} } dx\\
    	& \geq \frac{\omega_{N-1} }{N} c^N \left( \frac{R}{1 + R} \epsilon \right)^N \\
    	&+ \omega_{N-1} \left( \frac{R}{1+ R} \right)^{N + N \epsilon} \int_{ \epsilon}^{1} t^{N-1 + N \epsilon} \left( - \frac{N}{ \alpha_N c^{ \frac{1}{N-1} } } \right)^N \ln^N \left|t \right| dt\\
    	& +\omega_{N-1} \left( \frac{R}{1+ R} \right)^{N + N \epsilon} \int_{\epsilon}^{1} t^{N-1 + N \epsilon} \left( - \frac{N}{ \alpha_N c^{ \frac{1}{N-1} } } \right)^{ \frac{N^2}{N-1} } \ln^{ \frac{N^2}{N-1} } \left|t \right| dt.
    \end{flalign*}
    It suffices to prove
    \begin{flalign*}
    	\frac{ c^{ \frac{N}{N-1} } }{R^N } \rightarrow 0,
    \end{flalign*}
    as $\epsilon \rightarrow 0$.
    \begin{flalign*}
    	& \int_{ \mathbb{B} } |x|^{N \epsilon} \Phi \left( \alpha_N \left( 1+ \epsilon \right) \left|\frac{ u_{\epsilon}^{\prime \prime}   }{ \left\| u_{\epsilon} \right\|_{W^{1,N} } } \right|^{ \frac{N}{N-1} } \right) dx\\
    	& \geq \frac{\alpha_N^{N-1}}{ \left( N-1 \right)! } \left( 1- \frac{1}{ 1+ \left\| u_{\epsilon}^{  \prime \prime} \right\|_{L^N}^N } \right)  + \frac{\alpha_N^N }{ N! } \int_{\mathbb{B} \setminus \mathbb{B}_{\frac{R}{1+R} \epsilon} } |x|^{N \epsilon} |u_{\epsilon}^{ \prime \prime}|^{ \frac{N^2}{N-1} } dx\\
    	& \geq C_1 \left( \frac{R}{R+1} \right)^{ N - \frac{N^2}{N-1} } - C_2 \left( \frac{R}{R+1} \right)^{-N} c^{ \frac{N}{N-1} }\\
    	& = \frac{ c^{ \frac{N}{N-1} } }{ R^N } \left\lbrace C_1 \left( \frac{R}{R+1} \right)^{ N - \frac{N^2}{N-1} } \frac{ R^N }{c^{\frac{N}{N-1}} } - C_2 \left(R+1 \right)^N  \right\rbrace 
    \end{flalign*}
    Letting $R = b c^{ \frac{1}{N-2} }$, we have 
    \begin{flalign*}
    	& C_1 \left( \frac{R}{R+1} \right)^{ N - \frac{N^2}{N-1} } \frac{ R^N }{c^{\frac{N}{N-1}} } - C_2 \left(R+1 \right)^N \\
    	&= C_1 \left( \frac{R}{R+1} \right)^{ N - \frac{N^2}{N-1} } b c^{ \frac{N}{ \left(N-1\right) \left( N-2 \right) } } - C_2^{\prime} b^N c^{ \frac{N}{N-2} } \\
    	& >0.
    \end{flalign*}
    This completes the proof.
\end{proof}


\subsection*{Acknowledgment}
We would like to take this opportunity to thank the editor and the
reviewers again for their constructive comments and useful suggestions, and the time and eﬀorts
they have spent in the review process.

\section*{Declarations}

\noindent Consent to Participate declaration: not applicable

\noindent Consent to Publish declaration: not applicable

\noindent Ethics declaration: not applicable

\noindent Funding declaration: No funding was received for this work.

\end{document}